\documentclass{amsart}

\usepackage{hyperref}

\usepackage{amssymb} % For \twoheadrightarrow
\usepackage{cleveref}
\usepackage{color}
\usepackage{csquotes}
\usepackage{mathabx}
\usepackage{stmaryrd}
\usepackage{tikz-cd}
\usepackage{verbatim}
\usepackage{xy}

\xyoption{all}
\DeclareMathOperator{\ob}{{\rm ob\:}}
\DeclareMathOperator{\op}{{\rm op}}
\DeclareMathOperator{\TF}{{\rm TF}}
\DeclareMathOperator{\Tr}{{\rm Tr}}
\DeclareMathOperator{\Lie}{{\rm Lie}}
\DeclareMathOperator{\colim}{{\rm colim}}
\DeclareMathOperator{\Vect}{{\rm Vect}}
\DeclareMathOperator{\id}{{\rm id}}
\DeclareMathOperator{\Comm}{{\rm Comm}}
\DeclareMathOperator{\free}{{\rm free}}
\DeclareMathOperator{\Lan}{{\rm Lan}}
\DeclareMathOperator{\mor}{{\rm mor}}
\DeclareMathOperator{\cod}{{\rm cod}}
\DeclareMathOperator{\dom}{{\rm dom}}
\DeclareMathOperator{\disc}{{\rm disc}}
\DeclareMathOperator{\Monoids}{{\rm Monoids}}
\DeclareMathOperator{\Alg}{{\rm Alg}}

\theoremstyle{plain}
\newtheorem{prop}{Proposition}[section]
\newtheorem{lemma}[prop]{Lemma}
\newtheorem{theorem}[prop]{Theorem}
\newtheorem{corollary}[prop]{Corollary}
\newtheorem{question}[prop]{Question}

\theoremstyle{definition}
\newtheorem{definition-proposition}[prop]{Definition-Proposition}
\newtheorem{definition}[prop]{Definition}
\newtheorem{remark}[prop]{Remark}
\newtheorem{example}[prop]{Example}
\newtheorem{examples}[prop]{Examples}
\newtheorem{observation}[prop]{Observation}

\DeclareMathOperator{\grouplikes}{{\rm \Gamma}}
\DeclareMathOperator{\Susp}{{\rm Susp}}
\DeclareMathOperator{\TSusp}{{\rm TSusp}}
\DeclareMathOperator{\triv}{{\rm triv}}
\DeclareMathOperator{\forget}{{\rm forg}}
\DeclareMathOperator{\GPGen}{{\rm GPGen}}
\DeclareMathOperator{\PrimGen}{{\rm PrimGen}}
\DeclareMathOperator{\Hopf}{{\rm Hopf}}
\DeclareMathOperator{\Tor}{{\rm Tor}}
\DeclareMathOperator{\Ext}{{\rm Ext}}
\DeclareMathOperator{\Cocomm}{{\rm Cocomm}}
\DeclareMathOperator{\Rig}{{\rm Rig}}
\DeclareMathOperator{\Bialg}{{\rm Bialg}}
\DeclareMathOperator{\LeftSided}{{\rm LeftSided}}

\newcommand{\tred}[1]{\textcolor{red}{#1}}

\newcommand{\C}{\mathbb{C}}
\newcommand{\N}{\mathbb{N}}
\newcommand{\R}{\mathbb{R}}

\newcommand{\rtens}{\ {\color{red}\otimes} \ }
\newcommand{\rtensk}{\ {\color{red}\otimes}_k \ }
\newcommand{\btens}{\ {\color{blue}\otimes} \ }

\newcommand{\gtens}{\ {\color{green}\otimes} \ }
\newcommand{\gtensk}{\ {\color{green}\otimes}_k \ }

\newcommand{\LieFilt}{{}^{I}\!E^0}
\newcommand{\SLieFilt}{{}^{I}\!\tilde{E}^0}

\renewcommand{\epsilon}{\varepsilon}

\title{Milnor-Moore theorems for bialgebras in characteristic zero.}

\author{Joey Beauvais-Feisthauer, Yatin Patel, Andrew Salch}
\email{joeybf@wayne.edu}
\email{yatin@wayne.edu}
\email{asalch@wayne.edu}

\address{Department of Mathematics, Wayne State University, 656 W. Kirby,  Detroit, MI 48202}

\begin{document}

\begin{abstract}
Over fields of characteristic zero, we construct equivalences between certain categories of bialgebras which are generated by grouplikes and generalized primitives, and certain categories of structured Lie algebras. The relevant families of bialgebras include many which are not connected, and which fail to admit antipodes.
\end{abstract}

\maketitle
\tableofcontents

\section{Introduction}

\subsection{The Milnor-Moore theorem, and some generalizations.}

Let $k$ be a field of characteristic zero. The characteristic zero case of the classical theorem of Milnor and Moore \cite{MR0174052} tells us that the category of Lie $k$-algebras is equivalent to the category of primitively-generated Hopf $k$-algebras. 
The equivalence of categories is given by the two functors
\begin{align*}
 P: \PrimGen\Hopf\Alg(k) & \rightarrow \Lie(k) \\
 U: \Lie(k) &\rightarrow \PrimGen\Hopf\Alg(k),
\end{align*}
where $P$ sends a Hopf algebra to its Lie algebra of primitives, and where $U$ sends a Lie algebra to its universal enveloping algebra. 
While many generalizations of this theorem can be found in the literature\footnote{See for example \cite{MR2746042}, \cite{MR3438319}, \cite{MR1879927}, \cite{MR1860997},  \cite{MR2504663}, \cite{MR1813766}, and \cite{MR1927436}. Some of these references are written in terms of bialgebras or generalizations of bialgebras, but as far as we have been able to determine, these and all other currently-available generalizations of the Milnor-Moore theorem make an assumption of one kind or another which amounts, in the classical setting of {\em bialgebras} (rather than an abstract generalization of bialgebras in some categorical setting), to connectedness, hence admitting an antipode. Connectedness assumptions are quite reasonable if one is already willing to assume the existence of an antipode, since in the presence of an antipode, the connected part and the grouplikes can be treated separately. In the bialgebra setting, no such decomposition is available, so the grouplikes and primitives have to be treated in a ``mixed'' way, as the reader can see from the methods developed in this paper.}, we have not been able to find any generalizations to bialgebras which are not necessarily connected, and which do not necessarily admit antipodes, i.e., bialgebras which are not necessarily Hopf algebras. The purpose of this paper is to formulate and prove several such generalizations. We are particularly focused on the bialgebras with non-invertible grouplike elements. Such bialgebras necessarily cannot admit antipodes. While we think these generalizations of the Milnor-Moore theorem are of some interest for purely algebraic reasons, we also have some motivation from examples in algebraic topology; see the end of this introduction for the relevance to topology.

It is clear that a sufficiently na\"{i}ve attempt to generalize the Milnor-Moore theorem to bialgebras cannot possibly succeed. 
If we are to generalize the Milnor-Moore theorem to bialgebras, we must decide what category appears on the ``Lie side'' of the correspondence, and what category appears on the ``bialgebra side'' of the correspondence. Primitively-generated bialgebras, however, are Hopf algebras; see \cref{Injectivity and...}, below, for a proof (which is certainly not new). So the bialgebra side of the equivalence cannot simply be primitively-generated bialgebras.

We propose that the bialgebra side ought to be the {\em generalized-primitively-generated} bialgebras. Recall that a primitive in a bialgebra $A$ is an element $a\in A$ such that $\Delta(a) = a\otimes 1 + 1\otimes a$. Meanwhile, a grouplike in $A$ is an element $a\in A$ such that $\Delta(a) = a\otimes a$ and $\epsilon(a)=1$. If $Q$ is a grouplike element of $A$, we define a {\em $Q$-primitive}\footnote{This is a special case of a ``skew-primitive,'' an element $a$ such that $\Delta(a) = Q\otimes a + a\otimes Q^{\prime}$ for grouplikes $Q,Q^{\prime}$ of $A$. Skew-primitives are well-studied, appearing already in Bourbaki's volume on Lie theory \cite{MR1728312}.} to be an element $a\in A$ such that $Q\otimes a + a\otimes Q$. A {\em generalized primitive} is a linear combination of $Q$-primitives for various grouplikes $Q\in A$. Finally, a bialgebra is {\em generalized-primitively-generated} if it is generated, as an algebra, by grouplikes and generalized primitives\footnote{Generalized-primitively-generated bialgebras are automatically co-commutative; see Proposition \ref{gpg implies cocomm}. So to be generalized-primitively-generated is equivalent to being co-commutative and generated by grouplikes and skew-primitives. The condition of being generated by grouplikes and skew-primitives is a natural one, due to substantial interest in the Andruskiewitsch-Schneider conjecture \cite{MR1780094}, i.e., the conjecture that every finite-dimensional Hopf algebra over an algebraically closed field of characteristic zero is generated by grouplikes and skew-primitives.}. 

Now recall that, for a bialgebra $A$, the set of grouplikes $\grouplikes(A)$ of $A$ forms a monoid under multiplication. 
When we turn to the Lie side of our equivalence of categories, we will find that the relevant type of extra structure on the Lie algebras will depend on the structure of the monoid $\grouplikes(A)$. So it is convenient, on the bialgebra side of our equivalence of categories, to consider not only the generalized-primitively-generated bialgebras. Instead, we fix a commutative monoid $G$, and we want to work with bialgebras $A$ which are generalized-primitively-generated {\em and equipped with a choice of isomorphism $\Gamma(A)\cong G$}. We call bialgebras with such a choice of isomorphism {\em $G$-rigid bialgebras}; see Definition \ref{def of rigid bialgebra} for the full definition\footnote{To be clear, the definition of a $G$-rigid bialgebra includes the condition that the grouplikes $\Gamma(A)$ of $A$ are contained in the center of $A$. Bialgebras with noncentral grouplikes do not lie within the scope of the Milnor-Moore-type theorems proven in this paper.}.

Now we turn to the Lie side of our equivalence of categories. Fixing a commutative monoid $G$, we find that the set of generalized primitives of a $G$-rigid $k$-bialgebra $A$ is not only a Lie algebra, but in fact a {\em $G$-suspensive Lie $k$-algebra}. A $G$-suspensive Lie $k$-algebra, defined in Definition \ref{def of suspensive lie algebras}, is a $G$-graded Lie $k$-algebra which is furthermore equipped with an action of $G$ which respects the $G$-grading and such that the Lie bracket is $kG$-bilinear. \Cref{G-suspensive vector spaces} has some general remarks and discussion on the subject of this kind of ``suspensive algebra,'' i.e., familiar algebraic gadgets like Lie algebras but which are equipped with both a $G$-grading and a $G$-action. The idea of ``suspensive algebra'' is elementary, but we do not know of anywhere where it has been examined before.

In Definition-Proposition \ref{def of W}, we define the {\em universal $G$-rigid enveloping algebra} $WL$ of a $G$-suspensive Lie algebra $L$. In Proposition \ref{W and GP adjunction}, we find that 
\begin{align*}
W: \Susp_G\Lie(k) &\rightarrow \Cocomm\Rig_G\Bialg(k)
\end{align*}
 is left adjoint to the functor
\begin{align*}
GP_*: \Cocomm\Rig_G\Bialg(k) &\rightarrow\Susp_G\Lie(k)
\end{align*}
that sends a cocommutative $G$-rigid $k$-bialgebra to its $G$-suspensive Lie $k$-algebra of generalized primitives. Our first generalization of the characteristic zero Milnor-Moore theorem is Theorem \ref{mm thm char 0}, which states that $W$ and $GP$ restrict to an equivalence of categories between 
\begin{itemize}
\item {\em torsion-free} $G$-suspensive Lie $k$-algebras, and
\item {\em torsion-free} generalized-primitively-generated $G$-rigid $k$-bialgebras.
\end{itemize}
A $G$-suspensive Lie $k$-algebra $L$ is {\em torsion-free} if, for each nonzero homogeneous element $x\in L$ of degree $Q\in G$, we have $Q\cdot x\neq 0$. Similarly, a $G$-rigid bialgebra $A$ is {\em torsion-free} if its $G$-suspensive Lie algebra of generalized primitives is torsion-free. This curious kind of ``torsion'' is only defined because $G$-suspensive Lie algebras are equipped with {\em both} a $G$-grading and a $G$-action. To be clear, when $G$ is not a group, not every torsion-free $G$-suspensive Lie $k$-algebra is free (or projective) as a $kG$-module, so $G$-suspensive Lie algebras are not a special case of Lie $R$-algebras projective over $R$, as studied already in \cite{MR0174052}.

If the commutative monoid $G$ {\em is} a group, then the whole theory reduces to the well-known classical one: every $G$-suspensive Lie algebra is torsion-free, and furthermore Theorem \ref{mm thm char 0} is a case of the classical Milnor-Moore theorem. See Corollary \ref{tf corollary} for discussion. We emphasize that the results in this paper, and in particular Theorem \ref{mm thm char 0}, are only of consequence in the case when the monoid $G$ has noninvertible elements, i.e., the case of bialgebras with some noninvertible grouplikes.

Theorem \ref{mm thm char 0} does require the torsion-freeness hypothesis, as we explain with an explicit example in Example \ref{non-tf example}. In \cref{Suspensive mixed algebras.} we explain what would be necessary to formulate a {\em completely} general Milnor-Moore theorem for bialgebras, one that drops the torsion-freeness hypothesis. We explain that, in order for the ``bialgebra side'' of the equivalence to be {\em all} generalized-primitively-generated $G$-rigid bialgebras, the ``Lie side'' would have to be $G$-suspensive Lie algebras equipped with a truly ponderous structure, an $n$-ary product for each $n=2, 3, \dots$, satisfying some inconvenient associativity conditions as well as conditions enforcing compatibility with the Lie bracket and the suspensive structure. This structure is truly burdensome to work with, and we regard such a wide generalization of the Milnor-Moore theorem as unfruitful: one might as well just work with the bialgebras, if the Lie side of the equivalence is so complicated.

Nevertheless, there are other classes of bialgebras and suspensive Lie algebras, besides the torsion-free ones, for which a Milnor-Moore-type equivalence holds. In Theorem \ref{mm thm for left-sided bialgs} we prove a Milnor-Moore-type theorem for {\em torsion} $G$-suspensive Lie algebras, i.e., $G$-suspensive Lie algebras $L$ in which $Qx=0$ for all homogeneous $x\in L$ in degree $Q\in G$. Such suspensive Lie algebras are at the opposite extreme from those considered in Theorem \ref{mm thm char 0}, which were torsion-{\em free}. Theorem \ref{mm thm for left-sided bialgs} establishes that, when the commutative monoid $G$ is linear\footnote{Linearity of $G$ is defined in Definition \ref{def of linear}. The motivating example of a linear commutative monoid $G$ is the free commutative monoid $\mathbb{N}$, since this case appears prominently in motivating topological examples, like the Dyer-Lashof algebra.}, the category of torsion $G$-suspensive Lie $k$-algebras is equivalent to the category of {\em left-sided} generalized-primitively-generated $G$-rigid $k$-bialgebras. A $G$-rigid $k$-bialgebra $A$ is {\em left-sided} if, for all $Q,Q^{\prime}\in G$ such that $Q$ divides $Q^{\prime}$, for all $x\in GP_Q(A)$ and all $y\in GP_{Q^{\prime}}(A)$ we have $xy=0$. Left-sidedness is a very strong condition, but satisfied by the associated graded bialgebra of a natural filtration on the Dyer-Lashof algebra, as explained in Proposition \ref{dyer-lashof algebra is left-sided}.

Throughout, we work over a field of characteristic zero. Presumably it is possible to generalize the results in this paper to fields of positive characteristic, by keeping track of a restriction map on the ``Lie side'' of each equivalence. We have found ourselves already with enough to say in characteristic zero that we chose not to pursue the positive-characteristic case in this paper.

\subsection{Topological examples and motivations.}

There are two topological sources of non-Hopf bialgebras which motivated the authors to consider Milnor-Moore-like theorems for bialgebras:
\begin{enumerate}
\item For each prime number $p$, the $p$-primary {\em Dyer-Lashof algebra}, written $R$, is an $\mathbb{N}$-graded co-commutative $\mathbb{F}_p$-bialgebra of operations on the mod $p$ homology of infinite loop spaces. (There is a different Dyer-Lashof algebra $R$ for each prime $p$, but the choice of prime $p$ is suppressed from the notation.) The degree zero subring of $R$ is isomorphic to $\mathbb{F}_p[Q^0]$ with $Q^0$ a noninvertible grouplike element, so the bialgebra $R$ clearly cannot admit an antipode. See the appendix to this paper %, \cref{Review of the Dyer-Lashof algebra.}, 
for a review of basic properties of $R$, including generators and relations.

Basterra \cite{MR1732625} and Miller \cite{MR526673} have constructed spectral sequences whose input involves $\Tor$ and $\Ext$ groups over the Dyer-Lashof algebra. Basterra's spectral sequences, in particular, are one of very few available tools for the calculation of topological Andr\'e-Quillen cohomology groups of commutative ring spectra. While one can use Priddy's Koszul duality from \cite{MR265437} to obtain a complete description of $\Ext_R^{*,*}(\mathbb{F}_p,\mathbb{F}_p)$ (at $p=2$, in \cite{MR526673}, and all primes $p$ in \cite{MR711048}), nevertheless one may hope to arrive a new (hopefully useful!) perspective on the homological algebra of $R$-modules by understanding $R$ in terms of Lie algebras. 

Recall that, for each prime $p$, the Steenrod algebra $A$ is a co-commutative Hopf $\mathbb{F}_p$-algebra. It is not the case that $A$ is primitively-generated. In May's thesis \cite{MR2614527}, May filters $A$ so that its associated graded Hopf algebra $E_0A$ {\em is} primitively-generated. Consequently May gets a spectral sequence whose input is $\Ext$ over $E_0A$, and whose output is $\Ext$ over $A$, and since $E_0A$ is primitively-generated, May is able to use techniques from Lie algebra cohomology\footnote{The Lie-algebra-theoretic content of May's thesis is glossed over in some treatments, such as \cite{MR860042}, where the focus is on using May's spectral sequence for explicit calculations of relevance for topology. See \cite{MR0193126} and \cite{MR0185595} for published accounts of the Lie-algebra-theoretic techniques and results from May's thesis.} to understand and calculate $\Ext_{E_0A}$. 

Similarly, while the Dyer-Lashof algebra $R$ is not {\em generalized}-primitively-generated, if we filter it by powers of the ideal consisting of all elements in positive degrees, its associated graded bialgebra {\em is} generalized-primitively-generated\footnote{Note that, unlike the Steenrod algebra, the augmentation ideal of $R$ is not the ideal of all elements in positive degrees. For example, $1-Q^0$ is in the augmentation ideal, but also in degree zero.}. In fact $E_0R$ is, at $p>2$, a left-sided $\mathbb{N}$-rigid $\mathbb{F}_p$-bialgebra, so our Milnor-Moore-type theorem for left-sided rigid bialgebras, Theorem \ref{mm thm for left-sided bialgs}, {\em nearly} applies to $E_0R$. The only trouble is that Theorem \ref{mm thm for left-sided bialgs} is for fields of characteristic zero! If the Lie-algebra-theoretic methods in this paper can indeed shed any light on homological questions about the Dyer-Lashof algebra, it will have to wait until positive-characteristic versions of our theorems are proven.
\item
We now sketch one more topological application of the ideas in this paper.
Given a path-connected homotopy-associative unital $H$-space $X$, a classical result of Cartan and Serre \cite{MR46045} identifies the primitives in the rational homology $H_*(X; \mathbb{Q})$ of $X$ as the image of the rational Hurewicz map $\pi_*(X)\otimes_{\mathbb{Z}}\mathbb{Q}\rightarrow H_*(X; \mathbb{Q})$. The Lie bracket of primitives in $H_*(X; \mathbb{Q})$ then agrees with the Samelson product on homotopy groups. 
A classical application of the characteristic zero Milnor-Moore theorem then proves:
\begin{theorem} \label{mm for h-spaces 1} {\bf (Milnor-Moore; see appendix of \cite{MR0174052}.)} If $X$ is a path-connected homotopy-associative unital $H$-space, then the rational homology of $X$ is isomorphic to the universal enveloping algebra of the rational homotopy $\pi_*(X)\otimes_{\mathbb{Z}}\mathbb{Q}$, regarded as a Lie algebra via the Samelson product. 
\end{theorem}

It is natural to ask how to generalize Theorem \ref{mm for h-spaces 1} to handle a homotopy-associative $H$-space $X$ which is {\em not path-connected.} Here are some of the obstacles to formulating such a generalization:
\begin{description}
\item[Algebraic obstacle] If $X$ is not path-connected, then $H_0(X; \mathbb{Q})$ will typically have nontrivial grouplikes. Then $H_*(X; \mathbb{Q})$ will not be primitively-generated, so the classical characteristic zero Milnor-Moore theorem cannot offer a complete description of $H_*(X; \mathbb{Q})$ in terms of some kind of Lie-algebraic data. We might deal with the grouplikes in $H_0(X;\mathbb{Q})$ by using the Cartier-Gabriel-Kostant-Milnor-Moore theorem which identifies every cocommutative Hopf algebra over an algebraically closed field of characteristic zero as the twisted tensor product of a group algebra and the universal enveloping algebra of some Lie algebra. However, this requires the Hopf algebra to indeed be a {\em Hopf} algebra, not just a bialgebra: the antipode plays an essential role in the proof of the Cartier-Gabriel-Kostant-Milnor-Moore theorem. 

Our point is that, when $X$ is a non-connected homotopy-associative unital $H$-space but not an $H$-group, we really need some analogue of the characteristic zero Milnor-Moore theorem which applies to bialgebras without antipode. Such a theorem is the main result of this paper.
\item[Topological obstacle] In the proof of the Cartan-Serre theorem on primitives in $H_*(X;\mathbb{Q})$, an essential role is played by the fact that the rational $k$-invariants of a path-connected homotopy-associative unital $H$-space $X$ are all trivial, so the rational homotopy type of $X$ splits as a product of Eilenberg-Mac Lane spaces. See section 9.1 of \cite{MR2884233} for a nice modern exposition. If we lift the hypothesis that $X$ is path-connected, then we are no longer guaranteed to have such a splitting: for example, consider the case where $X$ is the free topological monoid on any space with a rationally nontrivial $k$-invariant. 

Our point is that there is a topological obstacle to identifying $H_*(X; \mathbb{Q})$ in terms of the Lie algebra $\pi_*(X)\otimes_{\mathbb{Z}}\mathbb{Q}$: one ought to formulate and prove some appropriate analogue of the Cartan-Serre theorem on primitives in $H_*(X;\mathbb{Q})$, in the case of $X$ not path connected. 
 Dealing with that obstacle goes beyond the scope of this paper, however.
\end{description}
\end{enumerate}

\subsection{Acknowledgments.} 

This paper has its origins in conversations that A. Salch had with Sean Tilson several years ago, motivated by the desire to better understand the input for Basterra's spectral sequences from \cite{MR1732625}. Salch thanks Tilson for those conversations and for his hospitality during a visit to Wuppertal. 

\section{$Q$-primitives.}

\subsection{Basic definitions.}

Throughout, let $k$ be a field, and let $A$ be a bialgebra over $k$ with coproduct $\Delta: A\rightarrow A\otimes_k A$ and augmentation $\epsilon: A\rightarrow k$. 
\begin{definition}\label{def of grouplike}
An element $Q\in A$ is {\em grouplike} if $\Delta(Q) = Q\otimes Q$ and $\epsilon(Q) = 1$.
\end{definition}
Definition \ref{def of grouplike} is classical and standard, as are the following facts:
\begin{enumerate}
\item The set of grouplike elements of $A$ forms a monoid under multiplication, since $ab$ is grouplike if $a$ and $b$ are each grouplike. 
\item If $A$ is furthermore a Hopf algebra, then the antipode $\chi: A\rightarrow A$ yields an inverse operation on the monoid of grouplike elements of $A$, making that monoid into a group.
\item A cocommutative bialgebra over an algebraically closed field (or, more generally, a pointed cocommutative bialgebra over any field) is a Hopf algebra if and only if its monoid of grouplikes forms a group. See Proposition 9.2.5 of \cite{MR0252485} for this result, and Lemma 8.0.1 of \cite{MR0252485} for a proof that algebraic closure of the ground field implies pointedness of the bialgebra.
\end{enumerate}

There exist non-Hopf bialgebras whose grouplikes are all invertible, hence form a group: see Example 2 in \cite{MR587928} for an example of a cocommutative bialgebra over the real numbers which is not a Hopf algebra, but whose only grouplike is $1$. Of course a non-Hopf cocommutative bialgebra %whose only grouplike is $1$ 
cannot be primitively-generated. %, by the classical Milnor-Moore theorem.

Definition \ref{def of Q-primitives} is not new: it is a natural idea, and it appears, with slightly different terminology and notation, in section 1.1 of chapter II of Bourbaki's book \cite{MR1728312}. However, Bourbaki very quickly restrict their attention to classical primitives, so the more general theory of generalized primitives does not get developed very far in \cite{MR1728312}. 
\begin{definition}\label{def of Q-primitives}\leavevmode
\begin{itemize}
\item We write $\grouplikes(A)$ for the monoid of grouplike elements of $A$.
\item Suppose $Q\in \grouplikes(A)$. An element $a\in A$ is {\em $Q$-primitive} if $\Delta(a) = a\otimes Q + Q\otimes a$.
\item An element $a\in A$ is a {\em homogeneous generalized primitive} if $a$ is $Q$-primitive for some $Q\in \grouplikes(A)$. 
\item An element $a\in A$ is a {\em generalized primitive} if $a$ is a $k$-linear combination of homogeneous generalized primitives. We write $GP_*(A)$ for the $k$-vector space of generalized primitives in $A$.
\item Suppose that the monoid $\grouplikes(A)$ is commutative\footnote{Perhaps this assumption is not really necessary, but if we do not assume it, then we wind up talking about graded vector spaces which are graded by a {\em noncommutative} monoid. That is a pretty exotic kind of grading, so it seems safest to avoid it. The only reason to assume that $\grouplikes(A)$ is commutative here is so that we do not have to talk about $G$-gradings for noncommutative $G$.}. For each $Q\in \grouplikes(A)$, we write $GP_Q(A)$ for the $k$-vector space of $Q$-primitives in $A$. Then \[ GP_*(A) = \bigoplus_{Q\in \grouplikes(A)} GP_Q(A),\] that is, we regard $GP_*(A)$ as a graded $k$-vector space, graded by the monoid $\grouplikes(A)$. 
\end{itemize}
\end{definition}

\subsection{Structure of the set of generalized primitives.}

In general, the commutator of two generalized primitives does not have to be a generalized primitive, so unlike classical primitives, generalized primitives in $A$ are not guaranteed to form a sub-Lie-algebra of $A$ under the commutator bracket! However, under some reasonable hypotheses, we {\em do} get that the commutator of generalized primitives is a generalized primitive, as we prove in Proposition \ref{commutator of gen prims}. The relevant hypotheses are laid out in Definition \ref{def of compatibility}:
\begin{definition}\label{def of compatibility}
We say that the bialgebra $A$ has {\em primitive-grouplike compatibility} if, for every pair of grouplikes $Q,Q^{\prime}$ in $A$, every $Q$-primitive $a\in A$, and every $Q^{\prime}$-primitive $a^{\prime}\in A$, we have:
\begin{equation}\label{pgc eq} 0 =  aQ^{\prime} \otimes Qa^{\prime} + Qa^{\prime}\otimes aQ^{\prime} - a^{\prime}Q \otimes Q^{\prime}a - Q^{\prime}a\otimes a^{\prime}Q \in A\otimes_k A
.\end{equation}
\end{definition}
\begin{example}
In order of increasing generality:
\begin{itemize}
\item Every commutative bialgebra has primitive-grouplike compatibility.
\item If every grouplike in $A$ is contained in the center of $A$, then $A$ has primitive-grouplike compatibility. When the grouplikes of a bialgebra $A$ are contained in the center of $A$, we say that {\em $A$ has central grouplikes}.
\item If every noncentral grouplike $Q$ in $A$ has the property that $Qa = 0$ for every generalized primitive $a\in A$, then $A$ has primitive-grouplike compatibility. We refer to this last condition as {\em strong primitive-grouplike compatibility}.
\end{itemize}

For example, the Dyer-Lashof algebra $R$ does not have central grouplikes\footnote{\Cref{Review of the Dyer-Lashof algebra.} reviews the basic properties of the Dyer-Lashof algebra, in case this may be useful to the reader.}: its noncentral grouplikes are the positive powers of $Q^0$. However, $R$ {\em does} have strong primitive-grouplike compatibility, since its generalized primitives are all in the two-sided ideal of elements in positive degree, and since $Q^0Q^i=0$ and $Q^0\beta Q^i$ for all $i>0$, by the excess relations in $R$.

However, $R$ is not generated by its grouplikes and generalized primitives. For example, the element $Q^2\in R$ is not in the subalgebra of $R$ generated by the grouplikes and generalized primitives. We can fix this situation by filtering $R$ by the $K$-adic filtration, where $K$ is\footnote{For simplicity, the rest of this paragraph is written under the assumption that $p=2$, but an analogous argument also works at odd primes.} the two-sided ideal $K = (Q^1, Q^2, Q^3, \dots)$ of elements in positive degree. With this filtration, the associated graded bialgebra $E_0R$ {\em does} have central grouplikes. This is because, for $n>0$, the Adem relations yield that $Q^nQ^0$ is a sum of products of the form $Q^rQ^s$ for {\em positive} $r,s$. Since $Q^nQ^0$ is a product of an element in $K$-adic filtration $1$ and an element of $K$-adic filtration $0$, while each of the terms $Q^rQ^s$ in the sum has $K$-adic filtration $2$, we get that the product $Q^n\cdot Q^0$ is zero in the associated graded bialgebra $E_0R$, for $n>0$. Meanwhile, $Q^0\cdot Q^n$ is zero already in $R$, by the excess condition. So:
\begin{itemize}
\item  $Q^0x=0=xQ^0$ in $E_0R$ if $x$ is homogeneous and of positive degree, while
\item $Q^0x=xQ^0$ if $x$ is homogeneous of degree $0$, since the degree $0$ subring of $E_0R$ is simply $\mathbb{F}_p[Q^0]$. 
\end{itemize}
Hence $E_0R$ has central grouplikes. 
% Note that the grouplikes in $E_0R$ are not contained in the center, since $Q^0Q^n$ generally does not vanish in $E_0R$! So our motivating example, $E_0R$, falls into the third type of example above (noncentral grouplikes $Q$ and generalized primitives $a$ satisfy $aQ=0$), but not the second type. So the level of generality of Definition \ref{def of compatibility} really is necessary for our main applications of these ideas. \tred{\em (Joey, can you include your argument for why $Q^0\cdot Q^n$ is zero here? -A.S.)}
%We restrict to the case $p = 2$ for concreteness, but our conclusion holds without change for odd primes as well. 
\end{example}

\begin{prop}\label{commutator of gen prims}
Suppose that $k$ is a field and $A$ is a $k$-bialgebra whose monoid $\grouplikes(A)$ of grouplikes is commutative. 
Then $A$ has primitive-grouplike compatibility if and only if the generalized primitives in $A$ are closed under the commutator bracket.
\end{prop}
\begin{proof}
Suppose that $a,b\in A$, and $Q_a,Q_b$ are grouplikes in $A$, and $a$ is a $Q_a$-primitive and $b$ is a $Q_b$-primitive. Then we have
\begin{align*}
 \Delta\left([a,b]\right)
  &= (a\otimes Q_a + Q_a \otimes a)(b\otimes Q_b + Q_b\otimes b) \\
  & \hspace{10mm}  - (b\otimes Q_b + Q_b\otimes b)(a\otimes Q_a + Q_a \otimes a) \\
%  &= ab\otimes Q_aQ_b + aQ_b \otimes Q_ab + Q_ab\otimes aQ_b + Q_aQ_b\otimes ab \\
%  &\hspace{10mm} - ba\otimes Q_bQ_a - bQ_a \otimes Q_ba - Q_ba\otimes bQ_a - Q_bQ_a\otimes ba \\
%  &= (ab - ba)\otimes Q_aQ_b + Q_aQ_b \otimes (ab - ba),
  &= aQ_b \otimes Q_ab + Q_ab\otimes aQ_b   - bQ_a \otimes Q_ba - Q_ba\otimes bQ_a \\
  &\hspace{10mm} + [a,b]\otimes Q_aQ_b + Q_aQ_b \otimes [a,b] 
\end{align*}
so $[a,b]$ is a $Q_aQ_b$-primitive if and only if the equation \eqref{pgc eq}, which defines primitive-grouplike compatibility, is satisfied.
\end{proof}

\begin{corollary}
If $A$ has primitive-grouplike compatibility, then $GP_*(A)$ is a sub-Lie-algebra of $A$, regarded as a Lie algebra via the commutator bracket.
Furthermore, $GP_*(A)$ is a $\grouplikes(A)$-graded Lie algebra over $k$. That is, given grouplikes $Q,Q^{\prime}$ in $A$, the Lie bracket of a $Q$-primitive and a $Q^{\prime}$-primitive lies in $GP_{Q\cdot Q^{\prime}}(A)$.
\end{corollary}

Suppose that $A$ has primitive-grouplike compatibility. Then the $\grouplikes(A)$-graded Lie algebra $GP_*(A)$ has additional structure: given an element $Q\in \grouplikes(A)$ and a $Q^{\prime}$-primitive $a\in A$, we have
\begin{align*}
 \Delta(Qa) 
  &= (Q\otimes Q)(Q^{\prime}\otimes a + a \otimes Q^{\prime}) \\
  &= QQ^{\prime}\otimes Qa + Qa \otimes QQ^{\prime},
\end{align*}
that is, $Qa$ is a $QQ^{\prime}$-primitive. So the monoid $\grouplikes(A)$ acts on the left on $GP_*(A)$ by $k$-linear endomorphisms which are grading preserving, in the sense that $Qa\in GP_{QQ^{\prime}}(A)$ if $a\in GP_{Q^{\prime}}(A)$.
For the same reasons, $\grouplikes(A)$ also has a right action on $GP_*(A)$.

If $A$ furthermore has central grouplikes, then the left and the right actions of $\grouplikes(A)$ on $GP_*(A)$ coincide, and we furthermore have
\begin{equation}
 [Qa,b] = Q[a,b] = [a,Qb]
\end{equation}
for any $Q\in \grouplikes(A)$, i.e.,
the Lie bracket on $GP_*(A)$ is $k[\grouplikes(A)]$-bilinear.

Evidently $GP_*(A)$ is endowed with various algebraic structures, more than simply the Lie bracket enjoyed by the classical primitives in $A$. This is especially true when $A$ has central grouplikes. Our next task, accomplished in Definition-Proposition \ref{def of W}, is to say precisely what kind of structured algebraic gadget $GP_*(A)$ is.

\begin{definition}\label{def of suspensive lie algebras}
Let $G$ be a commutative monoid, and let $k$ be a field. We use multiplicative notation for the monoid operation on $G$.
\begin{itemize}
\item By a {\em $G$-graded Lie algebra over $k$} we mean a Lie algebra $L$ over $k$ whose underlying $k$-vector space is equipped with a $G$-grading such that, if $x,y$ are homogeneous elements of $L$ such that $x\in L_m$ and $y\in L_n$, then $[x,y]\in L_{mn}$.
\item By a {\em $G$-suspensive %\footnote{We can probably come up with a better name than ``suspensive.''} 
Lie algebra over $k$} we mean a $G$-graded Lie algebra $L$ over $k$ equipped with
an action of $G$ on $L$ such that
\begin{itemize}
\item the $G$-action preserves gradings, in the sense that, if $g\in G$ and $x$ is a homogeneous element of $L$ with $x\in L_n$, then $g\cdot x\in L_{g\cdot n}$; and
\item the Lie bracket is $kG$-bilinear.
\end{itemize}
A {\em homomorphism of $G$-suspensive Lie algebras over $k$} is a homomorphism of the underlying Lie algebras which preserves the grading and commutes with the $G$-action. We denote the resulting category of $G$-suspensive Lie $k$-algebras by $\Susp_G\Lie(k)$.
\end{itemize}
\end{definition}
%If $G$ is a group, then $G$-suspensive Lie algebras are not very interesting: the $G$-action is forced to be free, so one might as well think of them as Lie $kG$-algebras which are free over $kG$. The theory of Lie $R$-algebras which are projective as $R$-modules is well-developed already in \cite{MR0174052}. However 
When $G$ fails to have inverses, then $G$-suspensive Lie algebras are much more interesting than when $G$ is a group. In particular, there are many examples where the $G$-action is far from free. Consequently the results of \cite{MR0174052} do not straightforwardly apply to such examples.

If $A$ is a $k$-bialgebra with central grouplikes, then the generalized primitives in $A$ form a $\Gamma(A)$-suspensive Lie $k$-algebra. A nice puzzle for the interested reader is to decide whether there is any further structure enjoyed by the generalized primitives in $A$. In other words:
\begin{question}\label{question 1}
The set of generalized primitives in $A$ has a $k$-vector space structure, a Lie bracket, a $G$-grading, and a $G$-action. Is there any other natural structure on the set of generalized primitives in $A$? 
\end{question}
In Corollary \ref{tf corollary} we find that the answer to Question \ref{question 1} turns out be ``no'' under certain hypotheses. These hypotheses are automatically satisfied if the monoid $G$ is a group, for example. However, in \cref{Suspensive mixed algebras.} we find that, if $G$ has noninvertible elements, then the answer to Question \ref{question 1} is ``yes'': the generalized primitives in $A$ form a very richly structured algebraic gadget. See \cref{Suspensive mixed algebras.} for discussion.

\begin{definition-proposition} \label{def of W} 
Let $k$ be a field. Let $G$ be a commutative monoid, and let $L$ be a $G$-suspensive Lie algebra over $k$. Let $W(L)$ be the $k$-bialgebra defined as follows. 
As a $k$-algebra,
\[ W(L) = \left( kG\otimes_k T(L)\right)/\left( 1\otimes (\ell\ell^{\prime} - \ell^{\prime}\ell) - 1\otimes [\ell,\ell^{\prime}],  g\otimes \ell - 1\otimes (g\cdot \ell)\right),\]
where $T(L)$ is the tensor $k$-algebra (i.e., free associative $k$-algebra) on $L$. 

%We now introduce some notation which might help to alleviate possible confusion. Since the bialgebra $W\left( L\right)$ is a quotient of $kG\otimes_k k\langle L\rangle$, we often specify elements of $W\left( L\right)$ by expressions like $g\otimes \ell$. When we begin to carry out algebraic manipulations involving the coproduct on $W\left( L\right)$, we often need to specify individual elements of $W\left( L\right)\otimes_k W\left( L\right)$, and in an expression like 
%\begin{equation}\label{expression 1} g_1\otimes \ell_1 \otimes g_2\otimes \ell_2,\end{equation} the first and third tensor symbols (which are the tensor products in $kG\otimes_k k\langle L\rangle$) play a different role than the second tensor symbol (which is really the same tensor product symbol appearing in $W\left( L\right)\otimes_k W\left( L\right)$). It makes our calculations and arguments clearer if we adopt a different notation for one tensor symbol than for the other. Consequently we will write $\odot$ in place of the first and third tensor symbols in the expression \ref{expression 1}, so to specify an element of $W\left( L\right)$ by an element $g\in G$ and an element $\ell\in L$, we will write $g\odot \ell$ rather than $g\otimes \ell$.

%Whenever it simplifies notation, we will write $\ell$ rather than $1\odot \ell$ for an element of $WL$.

In order to define the coproduct on $W(L)$, it is convenient to introduce a notational novelty: we will write $\odot$ for the {\em internal} tensor product in $W(L)$, and we will write $\otimes$ for the {\em external} tensor product in $W(L)\otimes_k W(L)$. So, to specify an element of $W(L)$, we can give a $k$-linear combination of elements of the form $g\odot \ell$, and to specify an element of $W(L)\otimes_k W(L)$, we can give a $k$-linear combination of elements of the form $\left( g\odot \ell\right) \otimes \left( g^{\prime}\odot \ell^{\prime}\right).$
With this notation in place, the coproduct on $W(L)$ is given by
\begin{align*}
  \Delta(g \odot 1) &= (g\odot 1) \otimes (g \odot 1), & g \in G, \\
  \Delta(1 \odot \ell) &= (1 \odot \ell) \otimes (Q \odot 1) + (Q \odot 1) \otimes (1 \odot \ell), & \ell \in L_Q.
\end{align*}
The augmentation on $W(L)$ is given by $\epsilon(g\odot 1) = 1$ and $\epsilon(1\odot\ell) = 0$ for all $g\in G$ and $\ell\in L$.

We call $W(L)$ the\footnote{While this definition of a ``universal $G$-rigid enveloping bialgebra'' makes sense without having a definition of a ``$G$-rigid bialgebra'' in general, we do provide a definition of $G$-rigid bialgebras below, in Definition \ref{def of rigid bialgebra}.} {\em universal $G$-rigid enveloping bialgebra of $L$}.
\end{definition-proposition}
\begin{proof}
It is routine to check that $W(L)$ is indeed a bialgebra, as follows. One verifies that $\epsilon$ and $\Delta$ are $k$-algebra morphisms by regarding $W(L)$ as a quotient of $kG$ tensored with the free associative $k$-algebra on $L$, and checking that $\Delta$ and $\epsilon$ each vanish on the relations $1\odot (\ell\ell^{\prime} - \ell^{\prime}\ell) - 1\odot [\ell,\ell^{\prime}]$ and $g\odot \ell - 1\odot (g\cdot \ell)$. Counitality and coassociativity are also straightforwardly verified.
\end{proof}

\begin{example}\label{U to W example}
Suppose that $G$ is the free monoid on a single generator $\sigma$.  
Let $L$ be the abelian $G$-suspensive Lie $k$-algebra consisting of a single copy of $k$ in each degree, with the action of $G$ freely permuting the copies of $k$. Then:
\begin{align*}
  UL &\cong k[x_0, x_1, x_2, \dots ], \\
  WL &\cong k[\sigma, x_0],
\end{align*}
where $x_n$ is the element $1\in k = L^{\sigma^n}$, the degree $\sigma^n$ summand of $L$.
The map 
\begin{align}
\label{U to W map} UL  &\rightarrow WL
\end{align} sends $x_n$ to $\sigma^nx_0$. While $x_n\in UL$ is a primitive for each $n$, its image $\sigma^nx_0$ in $WL$ is a $\sigma^n$-primitive. So, while the natural map \eqref{U to W map} is a $k$-algebra morphism, it does {\em not} preserve the coproduct. The map \eqref{U to W map} is also neither surjective (it does not hit $\sigma$, for example) nor injective.
\end{example}

\subsection{Generalized-primitive generation.}

It is classical that, given a Lie algebra $L$, its universal enveloping algebra $UL$ is primitively-generated. However, if $L$ is a $G$-suspensive Lie algebra, it is almost never the case that $WL$ is primitively-generated. Instead, $WL$ is easily seen to be {\em generalized-primitively-generated}, in the following sense:
\begin{definition}\label{def of gpg}
Let $A$ be a $k$-bialgebra. We say that $A$ is {\em generalized-primitively-generated} if every element of $A$ is a $k$-linear combination of products of grouplikes in $A$ and generalized primitives in $A$.
\end{definition} 
Note that being generalized-primitively-generated is a much weaker condition than being primitively-generated. For example, a bialgebra with no nonzero primitives, {\em and also no nonzero generalized primitives}, can still be generalized-primitively-generated, simply by being generated by grouplikes. This happens for group rings, for example. 

Nevertheless, not all bialgebras are generalized-primitively-generated. For example, non-co-commutative bialgebras cannot be generalized-primitively-generated, as we now show:
\begin{prop}\label{gpg implies cocomm}
If $A$ is a generalized-primitively-generated $k$-bialgebra, then $A$ is co-commutative.
\end{prop}
\begin{proof}
Straightforward: the coproduct in $A$ is commutative on grouplikes and on generalized primitives. So if grouplikes and generalized primitives generate $A$, then the coproduct is commutative on all elements of $A$.
%This is a very easy argument, but we prefer to write it out. If $a\in A$ is a grouplike, then $\Delta(a) = a\otimes a$. If $a$ is a $Q$-primitive, then $\Delta(a) = Q\otimes a + a\otimes Q$. In either case, we have that $\chi(\Delta(a)) = \Delta(a)$, where $\chi: A\otimes_k A \rightarrow A\otimes_k A$ is the factor swap (symmetry) map. Since $A$ is a $k$-bialgebra, $\Delta$ commutes with products and with $k$-linear combinations, so for every $k$-linear combination $a$ of grouplikes and generalized primitives in $A$, we have $\chi(\Delta(a)) = \Delta(a)$. But $A$ is generalized-primitively-generated, so every element of $A$ is indeed a $k$-linear combination of grouplikes and generalized primitives.
\end{proof}

The converse of Proposition \ref{gpg implies cocomm} fails, however. That is, it is not the case that every co-commutative $k$-bialgebra is generalized-primitively-generated. An example is as follows:
\begin{example} \label{example 230}
%A commutative bialgebra of the form $\R[C_p]^*$, where $C_p$ is any cyclic group of odd prime order $p$, is not generalized-primitively-generated. 
Let $C_3$ be the cyclic group of order three, and let $\R$ be the field of real numbers. Consider the linear dual $\R[C_3]^*$ of the group algebra $\R[C_3]$. 
By explicit, elementary calculation, one finds that
\begin{itemize}
\item the only grouplike in the Hopf algebra $\R[C_3]^*$ is $1$, 
\item so the only generalized primitives in $\R[C_3]^*$ must be primitives,
\item and there are no nonzero primitives in $\R[C_3]^*$,
\item so the subalgebra of $\R[C_3]^*$ generated by grouplikes and generalized primitives is $\R\subseteq \R[C_3]^*$.
\end{itemize}
Therefore $\R[C_3]^*$ is not generalized-primitively-generated. The argument extends to $C_n$ for any odd $n \geq 3$, because $\R$ is missing all odd-order roots of unity except for $1$.
\end{example}

The bialgebra of Example \ref{example 230} is cocommutative but not generalized-primitively-generated. However, after base change to the algebraic closure, it does become generalized-primitively-generated. This suggests the idea that perhaps cocommutative bialgebras over algebraically closed fields are generalized-primitively-generated, or more generally, that {\em pointed} cocommutative bialgebras are generalized-primitively-generated\footnote{See Definition \ref{def of coradical filt} for the definition of pointedness for bialgebras and coalgebras. See Definition \ref{def of grouplike} for a statement of the relationship between pointedness and algebraic closure of the ground field.}. This would be a kind of converse to Proposition \ref{gpg implies cocomm}. We know no reason to expect such an idea to actually be true, however: for example, the Dyer-Lashof algebra is cocommutative but fails to be generalized-primitively-generated, even after base change to $\overline{\mathbb{F}}_p$. It is probably a very difficult problem to find a reasonable and useful sufficient condition on pointed cocommutative coalgebras which ensures that they are generalized-primitively-generated, since this begins to resemble the Andruskiewitsch-Schneider conjecture, which states that every finite-dimensional pointed Hopf algebra over an algebraically closed field of characteristic zero is generated by grouplikes and skew-primitives. See Conjecture 1.4 of \cite{MR1780094} for the conjecture, and Theorem 5.5 of \cite{MR2630042} for recent progress on it.

\section{$G$-suspensive vector spaces.}
\label{G-suspensive vector spaces}

We continue to let $G$ be a commutative monoid. In Definition \ref{def of suspensive lie algebras}, we defined $G$-suspensive Lie algebras. Recall that a $G$-suspensive Lie algebra is a vector space with a $G$-grading, a $G$-action, and a Lie bracket, satisfying some compatibility conditions between these three pieces of structure. In this section we make some observations about slightly weaker objects, ``$G$-suspensive vector spaces,'' which have the $G$-grading and $G$-action but not the Lie bracket. See Definition \ref{def of susp vect space} for the full definition. %We think the observations in this section give some intuition for how ``suspensive algebra'' behaves, but this section is not at all essential to the rest of the paper, and can safely be skipped without missing out on anything necessary for this paper's main results.

First, we recall (e.g. as in Proposition 5.12 in Dwyer's chapter of \cite{MR1926776}) the definition of the {\em transport category} of a monoid:
\begin{definition}\label{def of Tr(G)}
By the {\em transport category of $G$} we mean the category $\Tr(G)$ whose set of objects is $G$ itself, and such that the set of morphisms $\hom_{\Tr(G)}(g,h)$ from $g$ to $h$ in $\Tr(G)$ is the set of elements $f$ of $G$ such that $fg=h$. The composite $j\circ f$ of $f\in \hom_{\Tr(G)}(g,h)$ with $j\in \hom_{\Tr(G)}(h,i)$ is defined to simply be the product $jf$ in $G$.
%
%The transport category of $G$ has the structure of a symmetric monoidal category as follows: the monoidal product $g\otimes h$ of two objects is simply the product $gh$ in $G$, while the monoidal product of two morphisms $f,j$ is, again, simply the product $fj$ in $G$. The braiding $g\otimes h \stackrel{\cong}{\longrightarrow} h\otimes g$ is simply the equality $gh = hg$.
\end{definition}

%Given appropriate monoidal categories $\mathcal{C}$ and $\mathcal{D}$, the Day convolution product yields the structure of a monoidal category on the category $\mathcal{D}^{\mathcal{C}}$ of functors $\mathcal{C}\rightarrow\mathcal{D}$. The monoid objects in $\mathcal{D}^{\mathcal{C}}$ are equivalent to lax monoidal functors $\mathcal{C}\rightarrow\mathcal{D}$. If $\mathcal{C}$ and $\mathcal{D}$ are symmetric monoidal, then the Day convolution is also symmetric monoidal, so we can speak of the {\em commutative} monoid objects in $\mathcal{D}^{\mathcal{C}}$; these are equivalent to {\em symmetric} monoidal functors $\mathcal{C}\rightarrow\mathcal{D}$.

\begin{definition} \label{def of susp vect space} Let $k$ be a field.
%\begin{itemize}
%\item 
By a {\em $G$-suspensive $k$-vector space} we mean a functor $\phi: \Tr(G) \rightarrow\Vect(k)$. 

Equivalently, a $G$-suspensive $k$-vector space consists of a $G$-graded $k$-vector space $V = \bigoplus_{g\in G} \phi(g)$ together with a unital, associative action of $G$, such that, if $y\in G$ and $x\in V$ is in degree $h\in G$, then $y\cdot x\in V$ is in degree $yh$.
%\item By a {\em $G$-suspensive $k$-algebra} we mean a monoid object in the monoidal category $\Vect(k)^{\Tr(G)}$.
%
%Equivalently, a $G$-suspensive $k$-algebra is a lax monoidal functor $\phi: \Tr(G) \rightarrow\Vect(k)$. 
%
%Equivalently, a $G$-suspensive $k$-algebra consists of a $G$-suspensive $k$-vector space $A$ equipped with the structure of a $G$-graded $k$-algebra, satisfying the condition that $(g_1g_2)xy = (g_1x)(g_2y)$ for all $g_1,g_2\in G$ and all $x,y\in A$.
%\item By a {\em $G$-suspensive commutative $k$-algebra} we mean a commutative monoid object in the symmetric monoidal category $\Vect(k)^{\Tr(G)}$.
%
%Equivalently, a $G$-suspensive $k$-algebra is a symmetric lax monoidal functor $\phi: \Tr(G) \rightarrow\Vect(k)$. 
%
%Equivalently, a $G$-suspensive $k$-algebra consists of a $G$-suspensive $k$-algebra whose algebra structure is commutative.
%\item  By a {\em $G$-suspensive $k$-coalgebra} we mean a comonoid object in the monoidal category $\Vect(k)^{\Tr(G)}$.
%\end{itemize}
\end{definition}

\begin{examples} For various commonly-occurring commutative monoids $G$, here are alternative descriptions of the category of $G$-suspensive $k$-vector spaces.
\begin{itemize}
\item Let $G = \langle \sigma\rangle$, the free commutative monoid on one generator $\sigma$. Then a $G$-suspensive $k$-vector space consists of a sequence $\phi(1),\phi(\sigma),\phi(\sigma^2),\dots$ of $k$-vector spaces together with a $k$-linear function
\begin{align}
\label{morphism 1} \sigma_{n,n+1}: \phi(\sigma^n) &\rightarrow \phi(\sigma^{n+1})
\end{align}
for each nonnegative integer $n$. (Of course there is also a morphism $\sigma_{n,m}: \phi(\sigma^n)\rightarrow\phi(\sigma^m)$ for each $n\leq m$, but this morphism is equal to the composite $\sigma_{m-1,m}\circ\sigma_{m-2,m-1}\circ\dots\circ\sigma_{n+1,n+2}\circ\sigma_{n,n+1}$, so it is determined by the morphisms of the type \eqref{morphism 1}.)

In other words, the category of $G$-suspensive $k$-vector spaces is equivalent to the category of sequences $V_0 \rightarrow V_1 \rightarrow \dots$ of $k$-vector spaces.
\item Let $G = \langle \sigma\mid \sigma^2\rangle$, the cyclic group with two elements. By a similar analysis, % Then a $G$-suspensive $k$-vector space consists of two $k$-vector spaces $\phi(1)$ and $\phi(\sigma)$, a $k$-linear function $\sigma_{1,\sigma}: \phi(1) \rightarrow \phi(\sigma)$, and a $k$-linear function $\sigma_{\sigma,1}: \phi(\sigma) \rightarrow \phi(1)$ such that $\sigma_{\sigma,1}\circ \sigma_{1,\sigma} = \id_{\phi(1)}$ and $\sigma_{1,\sigma}\circ \sigma_{\sigma,1} = \id_{\phi(\sigma)}$. This implies that $\sigma_{\sigma,1}$ is uniquely determined by $\sigma_{1,\sigma}$. %, since $\sigma_{\sigma,1} = \sigma_{1,\sigma}^{-1}$. Consequently 
the category of $G$-suspensive $k$-vector spaces is equivalent to the category of $k$-linear isomorphisms, i.e., the subcategory of the category of arrows in $\Vect(k)$ such that the arrow is an isomorphism.
\end{itemize}
\end{examples}

Here is one more perspective on what ``$G$-suspensive algebra'' is about. Given a $k$-vector space $V$, it is classical that the data of a $G$-grading on $V$ is equivalent to the data of a coassociative, counital map $V \rightarrow V\otimes_k kG$, i.e., the structure of a $kG$-comodule on $V$. If $V$ is also equipped with a $k$-linear action of $G$, then given a coassociative counital map $\psi: V \rightarrow V\otimes_k kG$, we have a $G$-action on the domain of $\psi$ and also the diagonal $G$-action on the codomain of $\psi$. To give the structure of a $G$-suspensive vector space on $V$ is equivalent to giving a choice of coassociative counital map $\psi: V \rightarrow V\otimes_k kG$ which commutes with the $G$-action. So: {\em  a $G$-suspensive vector space is a vector space equipped with an action of $kG$ and a compatible coaction of $kG$.}

\section{Rigid bialgebras.}
Throughout, we continue to let $G$ be a commutative monoid, and let $k$ be a field.

We begin with the notion of a {\em $G$-rigid bialgebra}, i.e., a bialgebra with central grouplikes which is furthermore equipped with a choice of isomorphism between its monoid of grouplikes and $G$. The definition, in more detail, is as follows:
\begin{definition}\label{def of rigid bialgebra}
By a {\em $G$-rigid $k$-bialgebra} we mean a $k$-bialgebra $A$ equipped with a homomorphism of $k$-bialgebras $\eta: kG\rightarrow A$ satisfying each of the following conditions:
\begin{itemize}
\item The monoid map $\grouplikes(kG)\stackrel{\cong}{\longrightarrow} \grouplikes(A)$ induced by $\eta$ is an isomorphism.
\item The image of $\eta$ lies in the center of $A$.
\end{itemize}
 We refer to the homomorphism $\eta$ as the {\em rigid unit map of $A$}. A {\em homomorphism of $G$-rigid $k$-bialgebras} is a homomorphism of the underlying $k$-bialgebras which commutes with the rigid unit maps.
We denote the resulting category of $G$-rigid $k$-bialgebras by $\Rig_G\Bialg(k)$, and the full subcategory of \emph{cocommutative} $G$-rigid $k$-bialgebras by $\Cocomm\Rig_G\Bialg(k)$.
\end{definition}
It is worth being explicit about this point: if $A$ is a $G$-rigid $k$-bialgebra, then the elements of $G$ are grouplike in $A$, and consequently are central. Hence $A$ is a $kG$-algebra. However, a $G$-rigid $k$-bialgebra 
 is in general {\em not} a $kG$-{\em bi}algebra. The failure of a $G$-rigid $k$-algebra to be a $kG$-bialgebra is visible in the action of $kG$ on the tensor product $A\otimes_kA$, as follows:
if $A$ is a $G$-rigid $k$-bialgebra, then for any $g\in G$ and any $a_1\otimes a_2\in A\otimes_k A$, we have $g\cdot (a_1\otimes a_2) = ga_1 \otimes ga_2$. This is because $g$ is grouplike in $kG$ and because the rigid unit map $kG\rightarrow A$ is a bialgebra map.
On the other hand, if $A$ were instead a $kG$-bialgebra, then we would have $\Delta(g) = g\otimes 1 = 1\otimes g = g(1\otimes 1)$, which is not equal to $g\otimes g$ unless $g=1$. 

\begin{prop}\label{W and GP adjunction}
The functor $W: \Susp_G\Lie(k)\rightarrow \Cocomm\Rig_G\Bialg(k)$ defined in Definition-Proposition \ref{def of W} is left adjoint to the functor $$GP_{*}: \Cocomm\Rig_G\Bialg(k)\rightarrow \Susp_G\Lie(k).$$
\end{prop}
\begin{proof}
Suppose that $L$ is a $G$-suspensive Lie $k$-algebra, and suppose that $A$ is a cocommutative $G$-rigid $k$-bialgebra. It is elementary to check that the function 
\begin{align*}
 \alpha: \hom_{\Cocomm\Rig_G\Bialg(k)}\left( WL,A\right) 
  &\rightarrow \hom_{\Susp_G\Lie(k)}\left( L,GP_*(A)\right) 
\end{align*}
is well-defined, where $\alpha$ is given by letting $(\alpha(f))(\ell) = f(\ell)$ for each morphism $f: WL \rightarrow A$ of $G$-rigid $k$-bialgebras.
\begin{comment}(GOOD, BUT DECIDED TO COMMENT IT OUT)
The morphism $\alpha(f)$ is well-defined since, if $Q\in G$ and $\ell\in L_Q \subseteq L \subseteq WL$, then $\Delta(f(\ell)) = f(\Delta(\ell)) = Q\otimes f(\ell) + f(\ell)\otimes Q$. That is, $f(\ell)\in A$ is indeed in $GP_*(A)\subseteq A$.

It is elementary to check that $\alpha(f)$ is a Lie algebra homomorphism.
\begin{align*}
 \alpha(f)[\ell_1,\ell_2] 
  &= f\left( [\ell_1,\ell_2]\right) \\
  &= f\left( \ell_1\ell_2 - \ell_2\ell_1\right) \\
  &= f(\ell_1)f(\ell_2) - f(\ell_2)f(\ell_1) \\
  &= \left[ f(\ell_1),f(\ell_2)\right] \\
  &= \left[ \alpha(f)(\ell_1),\alpha(f)(\ell_2)\right] .
\end{align*}
It is also straightforward to verify that $\alpha(f)$ is a {\em $G$-suspensive} Lie algebra homomorphism: by construction, $\alpha(f)$ preserves the $G$-grading, and a routine calculation shows that $\alpha(f)$ commutes with the $G$-action:
\[
(\alpha(f))(g\cdot \ell) 
  = f(g\cdot \ell) \\
  = g\cdot f(\ell) \\
  = g\cdot (a(f))(\ell).\]

We conclude that $\alpha$ is well-defined. 
\end{comment}
To show that $GP_*$ is right adjoint to $W$, all we need is an inverse to $\alpha$. Such an inverse is the function 
\begin{align*}
 \beta: \hom_{\Susp_G\Lie(k)}\left( L,GP_*(A)\right) 
   &\rightarrow \hom_{\Cocomm\Rig_G\Bialg(k)}\left( WL,A\right) \end{align*}
given on an element $f\in \hom_{\Susp_G\Lie(k)}\left( L,GP_*(A)\right)$ as follows:
by the universal property of the free associative $k$-algebra $T(L)$, there exists a unique $k$-algebra homomorphism $\tilde{f}: T(L)\rightarrow A$ such that $\tilde{f}(\ell) = f(\ell)\in GP_*(A)\subseteq A$ for each $\ell\in L$. Since the grouplikes in $A$ are central, %the image of the unit map $\eta: kG\rightarrow A$ is in the center of $A$, so 
we get a well-defined $k$-algebra homomorphism $\overline{f}: kG\odot_k T(L) \rightarrow A$ given by $\overline{f}(g\odot \ell) = g\cdot \tilde{f}(\ell)$. (See Definition-Proposition \ref{def of W} for the definition of the $\odot$ notation.) Since $\overline{f}(1\odot [\ell_1,\ell_2]) = \overline{f}(1\odot \ell_1)\overline{f}(1\odot \ell_2) - \overline{f}(1\odot \ell_2)\overline{f}(1\odot \ell_1)$ and since $\overline{f}(g\odot \ell) = \overline{f}\left(1\odot (g\cdot \ell)\right)$, the map $\overline{f}$ factors as the projection $kG\odot_k T(L) \twoheadrightarrow W(L)$ followed by a unique $kG$-algebra homomorphism $W(L)\rightarrow A$, which is the desired morphism $\beta(f)$. 

We need to verify that $\beta(f): WL\rightarrow A$ is not only a morphism of $kG$-algebras, but in fact a morphism of $G$-rigid $k$-bialgebras. That $\beta(f)$ commutes with the $G$-action is straightforward from the relation $g\odot \ell = 1\odot g\cdot \ell$ in $W(L)$. Consequently $\beta(f)$ commutes with the rigid unit maps. Since $W(L)$ is generated, as a $kG$-algebra, by the elements of $L\subseteq W(L)$, to check that $\beta(f)$ is a coalgebra morphism it suffices to check that $\Delta(\beta(f)(\ell)) = \left(\beta(f)\otimes \beta(f)\right) \Delta(\ell)$ for all $\ell\in L$.
This is straightforward: suppose that $\ell\in L$ is homogeneous of degree $Q$. Then we have:
\begin{align*}
 \Delta(\beta(f)(\ell)) 
  &= \Delta(f(\ell)) \\
  &= f(\ell)\otimes Q + Q\otimes f(\ell) \in A\otimes_{k} A \\
  &= (\beta(f)\otimes \beta(f))(\ell\otimes Q + Q\otimes \ell) \\ 
  &= (\beta(f)\otimes \beta(f))(\Delta(\ell)),
\end{align*}
as desired.

We have $\beta(\alpha(f))(\ell) = \alpha(f)(\ell) = f(\ell)$ and $\alpha(\beta(f))(\ell) = \beta(f)(\ell) = f(\ell)$ straightforwardly from the definitions of $\alpha$ and $\beta$, so $\alpha$ and $\beta$ are indeed mutually inverse.
\end{proof}

\begin{comment}(FINE, BUT HANDLED BY EARLIER EXAMPLE)
\begin{remark}
The coproduct on $WL$ differs from the coproduct on $UL$. This is easily observed by considering a homogeneous element $\ell\in L$ of degree $Q\in G$: we then have $\Delta(\ell) = Q\otimes \ell + \ell \otimes Q$ in $WL$, but $\Delta(f(\ell)) = 1\otimes \ell + \ell\otimes 1$. It is tempting to think that $WL$ and $UL$ ought to be isomorphic as $k$-algebras, although not as $k$-coalgebras, since at a glance it looks like we have an $k$-algebra homomorphism $f: WL \rightarrow UL$ which sends an element $g\odot \ell\in WL$ to $g\ell$, and which has inverse given by the map $UL \rightarrow WL$ sending $\ell$ to $1\odot \ell$. The problem is that $f$ is {\em not} a $k$-algebra homomorphism: if $Q\in G$ and $\ell,\ell^{\prime}\in L$, then $(1\odot \ell)\cdot (Q\odot \ell^{\prime}) = (Q\odot \ell)\cdot (1\odot \ell^{\prime})$ in $WL$, but we do not generally have $(\ell)\cdot (Q\ell^{\prime}) = (Q\ell)\cdot \ell^{\prime}$ in $UL$. 
\end{remark}
\end{comment}

\section{Injectivity and generalized primitives.}
\label{Injectivity and...}

%This section is about the relationship between injectivity (respectively, surjectivity) of a rigid bialgebra map $A\rightarrow B$, and injectivity (respectively, surjectivity) of its induced map of suspensive Lie algebras $GP_*A\rightarrow GP_*B$.

%\subsection{Injectivity and $GP_*$.}

A classical result, Proposition 3.9 of \cite{MR0174052}, establishes that a map of augmented\footnote{To avoid possible confusion, we point out that here we are using the phrase ``augmented coalgebra'' in the same way as how the phrase is used in \cite{MR0174052}, i.e., a coalgebra $A$ equipped with a suitable {\em unit} map $k\rightarrow A$.} $k$-coalgebras $A\rightarrow B$, with $A$ connected and $k$ a field, is injective if and only if its induced map $PA \rightarrow PB$ is injective.
We will need an analogue of that classical result which applies to certain $G$-rigid bialgebras. That analogue is Proposition \ref{GP and injectivity}, and it is an easy consequence of two classical results on coalgebras and one classical result on bialgebras. We recall the three results below, as Theorems \ref{heyneman-radford thm} and \ref{taft-wilson thm} and \ref{skew prim gen implies pointed}.
First, we recall a classical definition (see chapter 5 of \cite{MR1243637} for an excellent treatment of these ideas):
\begin{definition}\label{def of coradical filt}
Let $k$ be a field, and let $A$ be a $k$-coalgebra.
\begin{itemize}
\item Given grouplikes $g,h$ of $A$, a {\em $(g,h)$-primitive of $A$} is an element $x\in A$ such that $\Delta(x) = g\otimes x + x\otimes h$. An element $x\in A$ is {\em skew-primitive} if $x$ is $(g,h)$-primitive for some grouplikes $g,h$ of $A$.
\item The {\em coradical} of $A$, written $A_0$, is the sum of all simple coalgebras of $A$.
\item The {\em coradical filtration of $A$} is the filtration of $A$ by subcoalgebras
\[ A_0 \subseteq A_1 \subseteq A_2 \subseteq \dots \subseteq A\]
of $A$, defined by letting $A_n$ be the preimage of $A\otimes A_{n-1} + A_0\otimes A$ under the coproduct map $A\rightarrow A\otimes_k A$.
\item We say that $A$ is {\em pointed} if every simple subcoalgebra of $A$ is one-dimensional as a $k$-vector space. That is (see the paragraph following 5.1.5 of \cite{MR1243637}), $A$ is pointed if and only if the coradical of $A$ coincides with the $k$-linear span of the grouplike elements of $A$. 
\end{itemize}
\end{definition}

\begin{theorem} {\bf (Heyneman-Radford.)}\label{heyneman-radford thm} Suppose that $k$ is a field, and suppose that $A,B$ are $k$-coalgebras. Let $f: A\rightarrow B$ be a coalgebra morphism whose restriction to $A_1$ is injective. Then $f$ is injective.
\end{theorem}
\begin{proof}
The original 1974 paper of Heymeman and Radford is \cite{MR346001}. See Theorem 5.3.1 in \cite{MR1243637} for a presentation (with proof) in the context of related results.
\end{proof}

The statement of the Taft-Wilson theorem given below, in Theorem \ref{taft-wilson thm}, is only one part of the original Taft-Wilson result, but we will not need to make use of the other parts of the result in this paper.
\begin{theorem} {\bf (Taft-Wilson.)}\label{taft-wilson thm} Suppose that $k$ is a field, and suppose that $A$ is a pointed $k$-coalgebra. Then $A_1$ is $k$-linearly spanned by the grouplikes and the skew-primitives of $A$.
\end{theorem}
\begin{proof}
The original paper of Taft and Wilson, also from 1974, is \cite{MR338053}. It also appears, with proof, as Theorem 5.4.1 in \cite{MR1243637}, again usefully in the context of related results.
\end{proof}

We do not know the original reference for Theorem \ref{skew prim gen implies pointed}. We learned of it from \cite{MR1483071}, which in turn cites Lemma 1 of \cite{MR1261906}. We do not know of an earlier reference. At least in its form for Hopf algebras rather than bialgebas, it seems to be quite well-known (e.g. ``It is easy to see that any Hopf algebra generated by grouplike and skew-primitive elements is automatically pointed'' in \cite{MR1913436}).
\begin{theorem}\label{skew prim gen implies pointed}
Let $k$ be a field, and let $A$ be a $k$-bialgebra which is generated, as a $k$-algebra, by its grouplike and skew-primitive elements. Then $A$ is pointed.
\end{theorem}
\begin{proof}
See Corollary 1.7.3 of \cite{MR1483071}.
\end{proof}

Now here is an easy consequence\footnote{This consequence is not new: it is, for example, nearly the same as Proposition 3 of \cite{MR909627}.} of Theorems \ref{heyneman-radford thm}, \ref{taft-wilson thm}, and \ref{skew prim gen implies pointed}:
\begin{prop}\label{GP and injectivity}
Let $k$ be a field, let $G$ be a commutative monoid, and let 
$A,B$ be $G$-rigid $k$-bialgebras. Suppose that $A$ is generalized-primitively-generated.
Let $f: A \rightarrow B$ be a homomorphism of $G$-rigid $k$-bialgebras. 
Then $f$ is injective if and only if $GP_*f: GP_*A\rightarrow GP_*B$ is injective.
\end{prop}
\begin{proof}
Clearly the injectivity of $f$ implies the injectivity of $GP_*f$.
For the converse, suppose that $GP_*f$ is injective. Since $A,B$ are $G$-rigid and $f$ is a morphism of $G$-rigid bialgebras, $f$ is an isomorphism on grouplikes. Since $A$ is co-commutative (by Proposition \ref{gpg implies cocomm}), every skew-primitive in $A$ must be a generalized primitive. Since $A$ is generalized-primitively-generated, it is consequently generated by grouplikes and skew-primitives, so Theorem \ref{skew prim gen implies pointed} implies that $A$ is pointed. Hence, by Theorem \ref{taft-wilson thm}, $f$ is injective on $A_1$. Then Theorem \ref{heyneman-radford thm} immediately gives us that $f$ is injective.
\end{proof}

Proposition \ref{GP and injectivity} plays an important role in the proof of Theorem \ref{mm thm char 0}, below.

\section{Milnor-Moore for bialgebras in characteristic zero, torsion-free case.}

Definition \ref{def of lie filt} recalls the Lie filtration, defined by Milnor and Moore in Definition 5.12 of \cite{MR0174052}, and offers a generalization of it.
\begin{definition}\label{def of lie filt}
Let $k$ be a field.
\begin{itemize}
\item Given a Lie $k$-algebra $L$, the {\em Lie filtration on $UL$} is the increasing filtration of $UL$ by $k$-vector subspaces
\[ UL \supseteq \dots \supseteq F_2UL \supseteq F_1UL \supseteq F_0UL \]
given by:
\begin{itemize}
\item $F_0UL = k$, i.e, $F_0UL$ is the image of the unit map $k\rightarrow UL$,
\item and for each positive integer $n$, $F_nUL$ is the image of the multiplication map \[ L\otimes_k F_{n-1}UL \subseteq UL\otimes_k UL \rightarrow UL.\]
\end{itemize}
\item 
Given a commutative monoid $G$ and a $G$-suspensive Lie $k$-algebra $L$, by the {\em Lie filtration on $WL$} we mean the increasing filtration of $WL$ by $k$-vector subspaces
\[ WL \supseteq \dots \supseteq \tilde{F}_2WL \supseteq \tilde{F}_1WL \supseteq \tilde{F}_0WL \]
given by letting $\tilde{F}_nWL$ be the $kG$-linear span of the image of $F_nUL\subseteq UL$ in $WL$ under the map \[ UL \rightarrow WL = (kG\otimes_k UL)/\left( g\otimes \ell - 1\otimes (g\cdot\ell)\right).\]
\end{itemize}
Following Milnor and Moore, we write $\LieFilt UL$ for the associated graded of the Lie filtration on $UL$. Similarly, we write $\SLieFilt WL$ for the associated graded of the Lie filtration on $WL$.
\end{definition}
 
\begin{observation}\label{observation re assoc graded of lie filt}
Here are some observations about Definition \ref{def of lie filt}.
\begin{enumerate}
\item It is classical that the elements in Lie filtration $n$ in $UL$ are those which are expressible as a linear combination of products of $(\leq n)$-tuples in $L\subseteq UL$. Similarly, the elements in Lie filtration $n$ in $WL$ are the linear combinations of products of grouplikes and $(\leq n)$-tuples in $L\subseteq WL$.
\item It is straightforward that, if $f: L\rightarrow L^{\prime}$ is a homomorphism of $G$-suspensive Lie $k$-algebras, then $f$ is compatible with the Lie filtration, in the sense that $Wf$ sends elements of $\tilde{F}_nWL$ to elements of $\tilde{F}_nWL^{\prime}$. Consequently $f$ induces a homomorphism $\SLieFilt f: \SLieFilt WL \rightarrow \SLieFilt WL^{\prime}$ of associated graded bialgebras. 
\end{enumerate}
\end{observation}

Let $L$ be a Lie algebra. In section 5 of \cite{MR0174052}, Milnor and Moore define 
$L^{\#}$ to be the underlying $k$-vector space of $L$ equipped with the trivial Lie bracket. Then $\LieFilt UL$ is isomorphic to $\LieFilt U(L^{\#})$.
%Milnor and Moore then define $\mathcal{A}(L)$ to be the universal enveloping algebra of $L^{\#}$. 
We now consider an analogue in the suspensive setting:
\begin{definition}
Suppose we are given a commutative monoid $G$, a field $k$, and a $G$-suspensive Lie $k$-algebra $L$.
%\begin{itemize}
%\item Let $L^{\#}$ denote the abelian Lie $k$-algebra with the same underlying $k$-vector space as $L$.
%\item 
Let $L^{\tilde{\#}}$ denote the abelian $G$-suspensive Lie $k$-algebra with the same underlying $G$-suspensive vector space as $L$. That is, $L^{\tilde{\#}}$ is the $k$-vector space underlying $L$, with the same $G$-grading as $L$, the same $G$-action as $L$, and with zero Lie bracket.
%\item Let $\mathcal{A}(L)$ denote the $k$-bialgebra $U(L^{\#})$.
%\item Let $\mathcal{B}(L)$ denote the $G$-rigid $k$-bialgebra $W(L^{\tilde{\#}})$.
%\end{itemize}
\end{definition}

In Proposition \ref{PBW-like lemma 2} and elsewhere, when we deal with $\SLieFilt W(L)$, it pays to use terminology which clearly distinguishes the two gradings on $\SLieFilt W(L)$: there is a $G$-grading coming from the fact that $L$ is $G$-suspensive, and there is a $\mathbb{N}$-grading coming from the fact that $\SLieFilt W(L)$ is the associated graded of the Lie filtration. Whenever there is risk of confusion, we will refer to degrees in the first grading as {\em suspensive degree} and degrees in the second grading as {\em Lie degree.}

One form (see Theorems 5.15 and 5.16 of \cite{MR0174052}) of the classical Poincar\'{e}-Birkhoff-Witt theorem states that, if $L$ is a Lie algebra over a field $k$ of characteristic zero, then $\LieFilt UL$ is isomorphic to the symmetric $k$-algebra $S_k(L^{\#})$ on the underlying $k$-vector space $L^{\#}$ of $L$. Here is the relevant $G$-suspensive version:
\begin{prop}\label{PBW-like lemma 2}
Let $G$ be a commutative monoid, let $k$ be a field of characteristic zero, and let $L$ be a $G$-suspensive Lie $k$-algebra. 
Regard the underlying $G$-suspensive $k$-vector space $L^{\#}$ of $L$ as a $kG$-module via the action of $G$ on $L$.
Equip the symmetric $kG$-algebra $S_{kG}(L^{\#})$ with the structure of a $G$-rigid $k$-bialgebra, by letting the elements of $L^{\#}\subseteq S_{kG}(L^{\#})$ in suspensive degree $Q$ be $Q$-primitives, and letting the elements $G\subseteq kG\subseteq S_{kG}(L^{\#})$ in degree $0$ be grouplike. 
Then we have an isomorphism of $G$-rigid $k$-bialgebras
\[ \SLieFilt WL \cong S_{kG}(L^{\#}).\]
\end{prop}
\begin{proof}
Clearly $\SLieFilt WL \cong \SLieFilt W(L^{\#})$ by the definition of the Lie filtration on $WL$. In Lie degree $0$, $\SLieFilt W(L^{\#})$ is the image of the rigid unit map, while in Lie degree $1$, $\SLieFilt W(L^{\#})$ is a copy of $kG\otimes_{kG} L^{\#} \cong L^{\#}$. In Lie degree $n\geq 1$, $\SLieFilt W(L^{\#})$ consists of the $k$-linear combinations of formal products of $n$ elements of $L^{\#}$, modulo the relations enforcing multilinearity of the $G$-action, i.e.,
\begin{align} 
\label{multilin eq 1} Q\cdot (\ell_1\cdot \dots \cdot\ell_n) 
  &= (Q\ell_1)\cdot \ell_2\cdot \dots \cdot\ell_n \\
\label{multilin eq 2}  &= \ell_1\cdot (Q\ell_2)\cdot \dots \cdot\ell_n \\
\nonumber  &= \dots \\
\label{multilin eq 3}  &= \ell_1\cdot \ell_2\cdot \dots \cdot (Q\ell_n) .\end{align}
This is precisely the $n$th symmetric power of the $kG$-module $L^{\#}$.
\end{proof}

\begin{corollary}\label{PBW-like corollary}
Let $G$ be a commutative monoid, let $k$ be a field of characteristic zero, and let $L$ be a $G$-suspensive Lie $k$-algebra. 
Then the canonical map $L\rightarrow WL$ is injective. 
\end{corollary}
\begin{proof}
We have the commutative diagram of $k$-vector spaces
\[\xymatrix{
 L \ar[r] 
  & \LieFilt UL \ar[r]\ar[d]^{\cong} & \SLieFilt WL \ar[d]^{\cong} \\
  & \LieFilt U(L^{\#}) \ar[r]\ar[d]^{\cong} & \SLieFilt W(L^{\tilde{\#}}) \ar[d]^{\cong} \\
  & S_k(L^{\#}) \ar[r] & S_{kG}(L^{\tilde{\#}})
}\]
and the composite $L \rightarrow S_{kG}(L^{\tilde{\#}})$ is an isomorphism onto the (Lie) degree $1$ summand in $S_{kG}(L^{\tilde{\#}})$. So the composite $L \rightarrow \SLieFilt WL\rightarrow S_{kG}(L^{\tilde{\#}})$ is one-to-one. So $L \rightarrow \SLieFilt WL$ is one-to-one, so $L \rightarrow WL$ is one-to-one.
\end{proof}

\begin{definition}\label{def of torsion-free action}
Suppose we are given a commutative monoid $G$ and a field $k$. 
\begin{itemize}
\item We say that a $k$-linear action of $G$ on a $k$-vector space $V$ is {\em strictly torsion-free} if, for every $v\in V$ and every $Q\in G$ such that $Qv = 0$, we have $v=0$.
\item Suppose that $V$ is a $G$-suspensive $k$-vector space. We say that an element $x\in V$ in degree $Q$ is {\em torsion} if $Q\cdot x=0$.
\item We say that a $G$-suspensive $k$-vector space $V$ is {\em torsion-free}\footnote{As far as we know, this particular notion of torsion-freeness is not the same as that studied anywhere else in the literature, including that studied in \cite{MR1967707}.} if the only torsion element of $V$ is zero. That is, $V$ is torsion-free if and only if, for every $Q\in G$ and every $v\in V_Q$ such that $Qv=0$, we have $v=0$.
\item At the opposite extreme, we say that a $G$-suspensive $k$-vector space $V$ is {\em torsion} if every element of $V$ is torsion.
\end{itemize}
\end{definition}
Note that the definition of a strictly torsion-free action of $G$ requires only an action of $G$, while the definition of a torsion-free action of $G$ requires both an action of $G$ and a $G$-grading on $V$, i.e., a $G$-suspensive structure on $V$. 

It is clear that, if a $G$-suspensive $k$-vector space is strictly torsion-free, then it is also torsion-free. It is not difficult to come up with counterexamples to the converse claim.

%Note that neither a torsion-free $G$-action nor a strictly torsion-free action is the same thing as a faithful $G$-action, in the sense of representation theory. For example, if $G$ is not only a monoid but a {\em group}, then every action of $G$ on a $k$-vector space is strictly torsion-free, but there can certainly be non-faithful $k$-linear representations of $G$.

The action of $G$ on a torsion-free $G$-suspensive $k$-vector space $V$ may be far from free. The torsion-freeness condition enforces that $Q\in G$ acts faithfully on degree $Q$ in $V$, but $Q$ may act quite non-faithfully on other degrees in $V$. Consequently the remarks following Definition \ref{def of suspensive lie algebras} apply here as well: torsion-free $G$-suspensive Lie $k$-algebras are not necessarily free (or projective) as $kG$-modules, so general results about Lie $R$-algebras projective over $R$ (as in \cite{MR0174052}) do not apply to general torsion-free $G$-suspensive Lie $k$-algebras.

\begin{prop}\label{L and GPWL in torsion-free case}
Let $G$ be a commutative monoid, let $k$ be a field of characteristic zero, and let $L$ be a {\em torsion-free} abelian $G$-suspensive Lie $k$-algebra. 
Then the natural inclusion $L \rightarrow GP_*(WL)$ is an isomorphism.
\end{prop}
\begin{proof}
Since $L$ is abelian, by Proposition \ref{PBW-like lemma 2} we have $GP_*WL \cong GP_*S_{kG}(L)$. Choose a $k$-linear basis $\{ x_i: i\in I\}$ for $L$ consisting of elements homogeneous with respect to the $G$-grading.
A homogeneous element of $S_{kG}(L)$ of Lie degree $0$ is simply an element of $kG$. An element of $S_{kG}(L)$ of Lie degree $n>0$ is a polynomial of degree $n$ in the variables $\{ x_i: i\in I\}$, modulo the relations \eqref{multilin eq 1} through \eqref{multilin eq 3} enforcing multilinearity of the $G$-action. We use the symbol $\vec{x}$ to denote such a monomial in $S_{kG}(L)$.

Similarly, an element of $S_{kG}(L) \otimes_k S_{kG}(L)$ is describable as a sum of products of elements $Q^{\ell} = Q\otimes 1$ and $Q^r = 1\otimes Q$ and $\vec{x}^{\ell} = \vec{x}\otimes 1$ and $\vec{x}^{r} = 1\otimes \vec{x}$.

From this perspective, the coproduct on $S_{kG}(L)$ sends a polynomial $f(\vec{x})$ to 
\[ f\left(\left| \vec{x}\right|^{\ell}\cdot \vec{x}^r + \left| \vec{x}\right|^{r}\cdot \vec{x}^{\ell}\right),\]
i.e., the same polynomial $f$ but with each instance of $x_i$ replaced by $\left| x_i\right|^{\ell}\cdot x_i^r + \left| x_i\right|^{r}\cdot x_i^{\ell}$, where $\left| x_i\right|$ is the suspensive degree of $x_i$.
Consequently, if $f$ is a $Q$-primitive, then
\begin{align}
\label{Q two-sided eq 1} 
 Q^{\ell}f(\vec{x}^r) + Q^rf(\vec{x}^{\ell}) 
  &= f\left(\left| \vec{x}\right|^{\ell}\cdot \vec{x}^r + \left| \vec{x}\right|^{r}\cdot \vec{x}^{\ell}\right).
\end{align}
Applying the multiplication map $S_{kG}(L) \otimes_k S_{kG}(L) \rightarrow S_{kG}(L)$ to \eqref{Q two-sided eq 1}
yields the equation
\begin{align}
\label{Q two-sided eq 2} 
 2Q\cdot f(\vec{x})
  &= f\left(2\left| \vec{x}\right|\cdot \vec{x}\right).
\end{align}
Both the coproduct and the $G$-action on $S_{kG}(L)$ are homogeneous with respect to the Lie grading, so if \eqref{Q two-sided eq 2} is true, then it must be true of each Lie-degree-homogeneous summand on each side. In other words, if we let $f_n$ denote the sum of the monomials of Lie degree $n$ in $f$, then we must have 
\begin{align}
\label{Q two-sided eq 3} 
 2Q\cdot f_n(\vec{x})
  &= f_n\left(2\left| \vec{x}\right|\cdot \vec{x}\right) \\
\label{Q two-sided eq 3a} 
  &= 2^n\left| \vec{x}\right|^n f_n\left(\vec{x}\right) .
\end{align}

Since the $G$-action on $S_{kG}(L)$ respects the $G$-grading on $L$, the only way for \eqref{Q two-sided eq 3} and \eqref{Q two-sided eq 3a} to be satisfied is for $Q$ to be equal to $\left| \vec{x}\right|^n$. 
So we finally have 
\begin{align}
\label{Q two-sided eq 4} 
  0 &= (2^n-2)\left| \vec{x}\right|^n f_n.
\end{align}
Since $k$ is a field of characteristic zero, the only way for \eqref{Q two-sided eq 4} to be satisfied is for either $n$ to be equal to $1$, or for $f_n$ to be annihilated by $\left| \vec{x}\right|^n$. But $L$ is torsion-free, and $f_n(\vec{x})$ is in suspensive degree $\left| \vec{x}\right|^n$. So, unless $n=1$, the only way for $f_n$ to be annihilated by $\left| \vec{x}\right|^n$ is if $f_n=0$. 

Consequently the only nonzero $Q$-primitives in $W(L) \cong S_{kG}(L)$ are in Lie degree $1$, i.e., they are elements of $L\subseteq WL$.
\end{proof}

\begin{definition}\label{def of tf bialgebra}
Let $G$ be a commutative monoid and let $k$ be a field. A $G$-rigid $k$-bialgebra $A$ is said to be {\em torsion-free} if, for every $Q\in G$ and every $Q$-primitive $a\in A$ such that $Q\cdot a = 0$, we have $a=0$.
\end{definition}
In other words, a $G$-rigid $k$-bialgebra $A$ is torsion-free if and only if the $G$-suspensive $k$-vector space $GP_*A$ is torsion-free.

We now introduce some abbreviations for various categories of objects we have considered in this section, and functors between them:
\begin{itemize}
\item
We denote by $\GPGen\Rig_G\Bialg(k)$ the full subcategory of $\Cocomm\Rig_G\Bialg(k)$ whose objects are the generalized-primitively-generated bialgebras.
\item
Similarly, we denote by $\TF\GPGen\Rig_G\Bialg(k)$ the full subcategory of 
$\Cocomm\Rig_G\Bialg(k)$ whose objects are the torsion-free generalized-primitively-generated bialgebras.
\item
We denote by $\TF\Susp_G\Lie(k)$ the full subcategory of $\Susp_G\Lie(k)$ generated by the torsion-free $G$-suspensive Lie $k$-algebras.
%\item Let $$\tilde{W}: \TF\Susp_G\Lie(k)\rightarrow \TF\GPGen\Rig_G\Bialg(k)$$ denote the functor $$W: \Susp_G\Lie(k)\rightarrow \Cocomm\Rig_G\Bialg(k)$$ with its domain and codomain restricted to the indicated categories.
%\item  Similarly, let $$\tilde{GP}_*: \TF\GPGen\Rig_G\Bialg(k)\rightarrow \TF\Susp_G\Lie(k)$$ denote the functor $$GP_{*}: \Cocomm\Rig_G\Bialg(k)\rightarrow \Susp_G\Lie(k)$$ with its domain and codomain restricted to the indicated categories.
\end{itemize}
\begin{theorem}\label{mm thm char 0} Let $k$ be a field of characteristic zero, and let $G$ be a %{\em finite}  (WE DON'T SEEM TO ACTUALLY NEED FINITENESS ANYMORE!)
commutative monoid. 
Then the functors $GP_*$ and $W$, with their domains and codomains restricted as follows:
\begin{align}
\label{eq 09211} W: \TF\Susp_G\Lie(k)&\rightarrow \TF\GPGen\Rig_G\Bialg(k)\\
\label{eq 09212} GP_*: \TF\GPGen\Rig_G\Bialg(k)&\rightarrow \TF\Susp_G\Lie(k)
\end{align}
are mutually inverse. Consequently the category of torsion-free $G$-suspensive Lie $k$-algebras is equivalent to the category of torsion-free generalized-primitively-generated $G$-rigid $k$-bialgebras.
\end{theorem}
\begin{proof}
Let $L$ be a torsion-free $G$-suspensive Lie $k$-algebra. The canonical map $L \rightarrow GP_*(WL)$ is injective by Corollary \ref{PBW-like corollary}.
We need to show that it is also surjective. Suppose that $x$ is a $Q$-primitive in $WL$. Passing to the associated graded of the Lie filtration on $WL$, we have that $x$ represents a $Q$-primitive $\overline{x}$ in 
\begin{equation}
\label{isos 049}
\SLieFilt (WL) = \SLieFilt W(L^{\#}) =W(L^{\#}),
\end{equation}
with \eqref{isos 049} due to Proposition \ref{PBW-like lemma 2}. By Proposition \ref{L and GPWL in torsion-free case}, $\overline{x}$ is in Lie filtration $1$ in $W(L^{\#}) = \SLieFilt WL$. So $x$ is in Lie filtration $1$ in $WL$. 

Lie filtration $1$ in $WL$ consists of $k$-linear combinations of elements in $L$ and elements in the image of the rigid unit map $kG\rightarrow WL$. The elements of $L$ are certainly generalized primitives. If $x = \sum_g \alpha_gg + \sum_g \ell_g$ for some $\sum_g \alpha_gg\in kG$ and for some $\ell_g\in L_g$ for each $g\in G$, then since $x$ is assumed to be a $Q$-primitive, we have
\begin{align}
\label{eq 0091}  \sum_g \left( \alpha_g Q \otimes g + Q\otimes \ell_g + \alpha_g g\otimes Q + \ell_g \otimes Q\right) 
  &= Q \otimes x + x \otimes Q \\
\nonumber  &= \Delta(x) \\
\label{eq 0092}  &= \sum_g \alpha_g g\otimes g + \sum_g \left( g\otimes \ell_g + \ell_g\otimes g\right).  
\end{align}
Reading off the $Q\otimes Q$ terms from the left-hand side of \eqref{eq 0091} and from \eqref{eq 0092}, we have the equality
\begin{align*}
 2\alpha_Q Q\otimes Q 
  &= \alpha_Q Q\otimes Q,
\end{align*}
i.e., $\alpha_Q = 0$. 
For any $g\in G$ such that $g\neq Q$, reading off the $g\otimes g$ terms from the left-hand side of \eqref{eq 0091} and from \eqref{eq 0092} instead yields the equality
\begin{align*}
 0 &= \alpha_g g\otimes g,
\end{align*}
so $\alpha_g=0$. So all the $\alpha$ coefficients are zero, i.e., $x$ is an element of $L\subseteq GP_*(WL)$. Consequently the natural embedding $L\rightarrow GP_*WL$ is an isomorphism, i.e., $GP_*$ is a left inverse functor to $W$, with domains and codomains as in \eqref{eq 09211} and \eqref{eq 09212}. 

We also need to show that $GP_*$ is a right inverse functor for $W$, but this is straightforward: if $A$ is a $G$-rigid $k$-bialgebra, then the image of the natural map $W(GP_*A) \rightarrow A$ is the subalgebra of $A$ generated by the generalized primitives. So for $A$ to be generalized-primitively-generated is precisely for the natural map $W(GP_*A) \rightarrow A$ to be surjective. Since $GP_*(W(GP_*A)) \rightarrow GP_*A$ is an isomorphism, the map $W(GP_*A) \rightarrow A$ is injective after applying $GP_*$. So $W(GP_*A) \rightarrow A$ is injective, by Proposition \ref{GP and injectivity}. So $GP_*$ is both right inverse and left inverse to $W$, with domains and codomains as in \eqref{eq 09211} and \eqref{eq 09212}.
\end{proof}

If $G$ is a group, then every $G$-suspensive vector space is torsion-free, and every $G$-rigid $k$-bialgebra is torsion-free. Consequently, when $G$ is a group, Theorem \ref{mm thm char 0} reduces to:
\begin{corollary}\label{tf corollary}
Let $k$ be a field of characteristic zero, and let $G$ be an abelian group. 
Then the category of $G$-suspensive Lie $k$-algebras is equivalent to the category of generalized-primitively-generated $G$-rigid $k$-bialgebras. The equivalence is realized by the functors $GP_*$ and $W$.
\end{corollary}
Corollary \ref{tf corollary} is not very remarkable. If $G$ is an abelian group and if a $k$-bialgebra $A$ is co-commutative, pointed, and $G$-rigid, then $A$ is a Hopf $k$-algebra; see the remarks immediately following Definition \ref{def of grouplike}. Generalized-primitively-generated rigid bialgebras are automatically pointed, by Theorem \ref{skew prim gen implies pointed}. Kostant's theorem splits every co-commutative Hopf $k$-algebra as the smash product of the group algebra of its grouplikes with the irreducible component of $1$; see Theorem 8.1.5 of \cite{MR0252485} for a textbook treatment. Since $A$ is assumed to be $G$-rigid, the grouplikes in $A$ are in the center of $A$, so this smash product in fact is simply a tensor product over $k$. If $k$ furthermore has characteristic zero, then the irreducible component of $1$ in $A$ coincides with the sub-bialgebra of $A$ generated by the primitives. 

The conclusion is this: Corollary \ref{tf corollary} amounts to only a rephrasing of the classical isomorphism of $k$-bialgebras $A\cong UPA \otimes_k kG$ that we have when $G$ is an abelian group. Hence the only new cases of Theorem \ref{mm thm char 0} are those in which the monoid $G$ has noninvertible elements. 

However, for topological applications, Theorem \ref{mm thm char 0} does not go quite as far as one would like: we want to be able to handle suspensive Lie algebras and rigid bialgebras which are {\em not} torsion-free. For example, the associated graded bialgebra $E_0R$ of the $K$-adic filtration on the Dyer-Lashof algebra is one of our motivating examples of a non-Hopf bialgebra. The bialgebra $E_0R$ is $\mathbb{N}$-rigid, but it is very far from being torsion-free. 

Omitting the adjective ``torsion-free'' from Theorem \ref{mm thm char 0} results in a false claim. Here is a counterexample to the statement of Theorem \ref{mm thm char 0} with the torsion-freeness hypothesis omitted:
\begin{example}\label{non-tf example}
Let $G$ be the free monoid generated by a single element $Q$. That is, $G$ is isomorphic to $\mathbb{N}$, but we write $G$ multiplicatively, as the monoid of nonnegative powers $\{ 1, Q, Q^2, Q^3, \dots\}$ of $Q$. Let $k$ be a field of characteristic zero, and let $L$ be the abelian $G$-suspensive Lie $k$-algebra which is one-dimensional as a $k$-vector space, concentrated in degree $Q$. The $G$-action on $L$ is necessarily trivial. In particular, the $G$-suspensive Lie $k$-algebra $L$ is not torsion-free.

Now consider the $G$-rigid $k$-bialgebra $WL$. It is isomorphic to\footnote{The notation $\odot$ was defined in Definition-Proposition \ref{def of W}.}
\begin{align*} 
 (kG\odot_kUL)/\left( g\odot \ell - 1\odot (g\cdot \ell)\right) 
  &\cong k[Q,x]/(Qx).\end{align*}
Of course $x$ is $Q$-primitive in $WL$. But consider $x^2$:
\begin{align*}
 \Delta(x^2) 
%  &= (x\otimes Q + Q\otimes x)^2 \\
  &= x^2 \otimes Q^2 + 2Qx  \otimes Qx + Q^2\otimes x^2 \\ 
  &= x^2 \otimes Q^2 + Q^2\otimes x^2 ,
\end{align*}
so $x^2$ is a $Q^2$-primitive in $WL$. A similar argument shows that all the positive powers of $x$ in $WL$ are generalized primitives. So $GP_*WL$ is an infinite-dimensional $G$-suspensive Lie $k$-algebra, and $L\rightarrow GP_*WL$ is not surjective. So $GP_*$ and $W$ can fail to be mutually inverse, if we allow non-torsion-free $G$-suspensive Lie algebras.  
\end{example}
In the next section we consider what additional structure on suspensive Lie algebras one would have keep track of, in order to have a completely general analogue of Theorem \ref{mm thm char 0} which would apply to {\em all} generalized-primitively-generated rigid bialgebras.

\section{The prospects for a completely general Milnor-Moore theorem for bialgebras.}
\label{Suspensive mixed algebras.}

In Example \ref{non-tf example}, we saw a $G$-suspensive Lie algebra whose canonical map $L\rightarrow GP_*(WL)$ failed to be surjective. This failure of surjectivity came about because $L$ contained a torsion element, i.e., an element $x\in L_Q$ such that $Q\cdot x = 0$. 

Perhaps it is becoming clear that, in a rigid bialgebra $A$ which isn't torsion-free, the collection of generalized primitives is richly structured:
\begin{itemize}
\item If $x\in GP_Q(A)$ and $Qx=0$, then $x^2$ is in $GP_{Q^2}(A)$. So there is a squaring operation on the torsion elements of $GP_*(A)$.
\item More generally, if $x_1\in GP_{Q_1}(A)$ and $x_2\in GP_{Q_2}(A)$ and $Q_1x_2=0$, then $x_1x_2\in GP_{Q_1Q_2}(A)$.
\item More generally, if $x_1\in GP_{Q_1}(A)$ and $x_2\in GP_{Q_2}(A)$ and $Q_1x_2\otimes Q_2x_1 + Q_2x_1\otimes Q_1x_2 =0$ in $A\otimes_k A$, then $x_1x_2\in GP_{Q_1Q_2}(A)$. So there is a product operation on $GP_*(A)$ defined on those pairs $( x_1,x_2)$ on which the operation
\begin{align*}
 s_2(x_1 \otimes x_2) &= \left| x_1\right|x_2 \otimes \left|x_2\right| x_1 + \left| x_2\right|x_1 \otimes \left|x_1\right| x_2\end{align*}
vanishes, where $\left| x\right|$ denotes the degree of $x$ in $GP_*(A)$.
\item More generally, we have a $n$-ary product operation defined on $n$-tuples $(x_1,\dots ,x_n)$ of elements in $GP_*A$ which are in the kernel of the $k$-linear function $s_n: L^{\otimes_k n} \rightarrow L^{\otimes_k 2}$ defined as follows:
\begin{itemize}
%\item Given a homogeneous element $\ell\in L$, let $\left| \ell\right|\in G$ be the degree of $\ell$.
\item Given a subset $U\subseteq \{ 1, \dots ,n\}$, an element $i\in \{ 1, \dots ,n\}$, and an element $\ell\in L$, write $x(U, \ell,i)$ for the element of $WL$ given by letting $x(U,\ell,i)$ be $\left|\ell\right|$ if $i\in U$, and letting $x(U,\ell,i)$ be $\ell$ if $i\notin U$.
\item 
Let $s_n(\ell_1 \otimes \dots \otimes \ell_n)$ be the element
\[ \sum_{U\subseteq \{ 1, \dots ,n\}: 0<\left| U\right| <n} \left( \prod_{i=1}^n x(U,\ell,i)\right)\otimes \left( \prod_{i=1}^n x(U^{\prime},\ell,i)\right)\]
of $WL\otimes_k WL$, where $U^{\prime}$ is the complement of $U\subseteq \{ 1,\dots ,n\}$.
\end{itemize}
If $(x_1,\dots ,x_n)\in \ker s_n$, then $x_1\cdot \dots \cdot x_n$ is a generalized primitive in $A$.
\end{itemize}
The reason that these strange functions $s_n$ arise is that, if $x_1,\dots  ,x_n\in A$ have the property that $x_i$ is a $\left| x_i\right|$-primitive for each $i$, then $s_n(x_1\otimes \dots\otimes x_n)$ is equal to the difference
\[ \Delta(x_1 \cdots x_n) - \left| x_1 \cdots x_n\right| \otimes (x_1\cdots  x_n) -  (x_1\cdots  x_n) \otimes \left| x_1 \cdots x_n\right|.\]
Hence the vanishing of $s_n(x_1\otimes \dots\otimes x_n)$ is equivalent to the product of the generalized primitives $x_1, \dots ,x_n$ also being a generalized primitive.

To give a full theory of structured Lie algebras which is equivalent (over a field of characteristic zero) to generalized-primitively-generated rigid bialgebras, those Lie algebras would need to be equipped with a large amount of structure: for each element $x\in \ker s_n\subseteq WL^{\otimes_k n}$, we would need to record a product element $m_n(x)\in L$. We could then impose the relation $m_n(x)-x$ on $WL$ to get a quotient bialgebra of $WL$ in which the ``spurious'' (i.e., not in the image of $L\hookrightarrow WL$) generalized primitives in $WL$ are identified with elements in Lie filtration $1$, i.e., elements in the image of $L\hookrightarrow WL$. 

But this is a tall order. It means keeping track of an $n$-ary operation $m_n$ on $L$ for each $n\geq 2$. It would not be enough to confine our attention to the case $n=2$, i.e., the generalized primitives in $WL$ which are linear combinations of products of pairs of elements in $L$. Recording only the data of such linear combinations of products of pairs would only enable us to impose the correct relations on $WL$ to quotient out the generalized primitives in Lie filtration $2$. To capture the ``spurious'' %(i.e., not in the image of $L\hookrightarrow WL$) 
generalized primitives in $WL$ which are in Lie filtration $n$, we need the data encoded by the $n$-ary product operation $m_n:\ker s_n\rightarrow L$. 

Recording the data of these $n$-ary product operations $m_2,m_3,\dots$, and axiomatizing the various properties and compatibilities that they satisfy, involves some ugly bookkeeping. While the authors worked out some of the resulting theory and surmise that it can be made to work, we do not feel that it winds up being valuable, because the resulting structured Lie algebras are such a headache to use that one is better off just working with the bialgebras.

This is made clear already in one of our motivating examples, the associated graded bialgebra $E_0R$ of the $K$-adic filtration on the Dyer-Lashof algebra $R$, where $K$ is the ideal of $R$ generated by all homogeneous elements in positive degree. For simplicity, consider the case $p=2$. Then $Q^0Q^n = 0 = Q^nQ^0$ for all $n>0$, and consequently $\Delta(x) = (Q^0)^n\otimes x + x\otimes (Q^0)^n$ for all $x$ in internal degree $n>0$ in $E_0R$. In other words, {\em every homogeneous element of positive internal degree in $E_0R$ is a generalized primitive.} Furthermore, the vanishing of $Q^0x$ for all $x$ in positive degree means that the operations $s_2,s_3,\dots$ vanish completely on elements of $E_0R$ in positive suspensive degrees. Consequently, to use the theory suggested in the previous paragraph, the products $m_2,m_3, \dots$ would need to record the full data of {\em all} of the products of homogeneous elements of positive internal degree in $E_0R$. This is a ridiculous situation: rather than the Lie algebra of generalized primitives offering a simpler and more familiar algebraic structure than the bialgebra, we find ourselves needing to supplement the Lie bracket on $E_0R$ with essentially the full structure of all multiplications on $E_0R$, encoded in an unfamiliar way! We might as well have stuck with $E_0R$ itself.

Our conclusion is that we do not have much confidence in the possibility of a {\em useful} Milnor-Moore theorem for {\em arbitrary} rigid bialgebras over a field of characteristic zero. However, for {\em some} families of rigid bialgebras (e.g. the torsion-free rigid bialgebras), we saw already in Theorem \ref{mm thm char 0} that a useful and satisfying Milnor-Moore theorem can indeed be obtained. So our next task is to formulate and prove a useful Milnor-Moore theorem for a family of bialgebras including $E_0R$ and others which structurally resemble $E_0R$.

\section{Milnor-Moore for left-sided bialgebras in characteristic zero.}

The associated graded bialgebra $E_0R$ of the $(Q^1,Q^2,\dots)$-adic filtration on the Dyer-Lashof algebra $R$ is a generalized-primitively-generated $\mathbb{N}$-rigid bialgebra, but since it is not torsion-free, Theorem \ref{mm thm char 0} does not apply to it. Indeed, {\em every} generalized primitive in $E_0R$ is torsion\footnote{See Definition \ref{def of torsion-free action} for the definition of torsion elements in suspensive Lie algebras.}.

Despite its extreme failure to be torsion-free, $E_0R$ has a special property which makes it much better suited to a Lie-algebra-theoretic analysis than many other non-torsion-free rigid $\mathbb{N}$-bialgebras. The essential property is the following:
\begin{prop}\label{dyer-lashof algebra is left-sided}
Let $x,y\in E_0R$ be homogeneous elements of internal degrees $\left|x\right|$ and $\left| y\right|$, respectively. Suppose $\left|x\right|$ and $\left| y\right|$ are each positive.
\begin{itemize}
\item If $p>2$ and $\left|x\right| \leq \left| y\right|$, then $xy=0$.
\item If $p=2$ and $\left|x\right| < \left| y\right|$, then $xy=0$.
\end{itemize}
\end{prop}
\begin{proof}
This is simply a consequence of the fact that all monomials of negative excess are zero in $R$; see \cref{Review of the Dyer-Lashof algebra.}.
\end{proof}

\begin{definition}\label{def of left-sided}
Let $G$ be a commutative monoid.
\begin{itemize}
\item Given $Q,Q^{\prime}\in G$, we say that {\em $Q$ divides $Q^{\prime}$} if there exists some $g\in G$ such that $gQ = Q^{\prime}$.
\item Suppose that $k$ is a field and that $A$ is a $G$-rigid $k$-bialgebra. We say that $A$ is {\em left-sided} if, for all $Q,Q^{\prime}\in G$ such that $Q$ divides $Q^{\prime}$, for all $x\in GP_Q(A)$ and all $y\in GP_{Q^{\prime}}(A)$ we have $xy=0$.
\item A {\em morphism of left-sided $G$-rigid $k$-bialgebras} is simply a morphism of $G$-rigid $k$-bialgebras whose domain and codomain are each left-sided. We write $\LeftSided\Rig_G\Bialg(k)$ for the resulting category of left-sided $G$-rigid $k$-bialgebras.
\end{itemize}
\end{definition}

The advantage of working with left-sided bialgebras is that, even when they fail to be torsion-free (in the sense of Definition \ref{def of tf bialgebra}), there is no need to bring in the arcane product structure $m_2,m_3, \dots$ from \cref{Suspensive mixed algebras.}. Use of the products $m_2, m_3, \dots$ is unnecessary since in a left-sided bialgebra, the product of two (or finitely many) generalized primitives is describable entirely in terms of the Lie bracket. That is, 
\begin{align*}
 xy &= \left\{ \begin{array}{ll} [x,y] &\mbox{\ if\ } \left| x\right| >\left| y\right| \\
 -\left[y,x\right] &\mbox{\ otherwise.}\end{array}\right.
\end{align*}

As a consequence of Proposition \ref{dyer-lashof algebra is left-sided}, $E_0R$ is left-sided for primes $p>2$. If $p=2$, then we have some nonzero products of elements in $E_0R$ of the same internal degree (e.g. $Q^1Q^1$). So while $E_0R$ is not left-sided for $p=2$, all such nonzero products arise from the squaring operation in $E_0R$. Consequently the full product structure of $E_0R$ is recoverable even when $p=2$ by a combination of the Lie bracket and the restriction map on the generalized primitives. The presence of the restriction is a feature of generalized primitives in bialgebras over fields of positive characteristic. Ultimately we would like to develop the theory of the present paper (including left-sided bialgebras) in positive characteristic, for the sake of understanding Lie-algebraic aspects of the Dyer-Lashof algebra. However, in this paper we confine ourselves to the characteristic zero case, where there is already much to be said.

The definition of left-sidedness for a rigid bialgebra $A$ is a condition on the vanishing of certain products of generalized primitives. This condition also has consequences for other products. For example, Proposition \ref{left-sided products of grouplikes and gps} establishes that certain products of grouplikes and generalized primitives vanish as well:
\begin{prop}\label{left-sided products of grouplikes and gps}
Let $G$ be a commutative monoid, and let $k$ be a field of characteristic not equal to $2$. Then the following are each true:
\begin{enumerate}
\item In a left-sided $G$-rigid $k$-bialgebra, every generalized primitive is torsion.
\item If $V$ is a torsion $G$-suspensive $k$-vector space, then $V_Q = 0$ for each element $Q\in G$ which has an inverse.
\end{enumerate}
\end{prop}
\begin{proof}\leavevmode
\begin{enumerate}
\item
Let $Q\in G$, and let $x\in A$ be a $Q$-primitive.
We have $x^2 =0$ by the definition of left-sidedness. Now apply the coproduct:
\begin{align*}
 0 
  &= \Delta(x^2) \\
  &= \left( Q\otimes x + x\otimes Q\right)^2 \\
  &= Q^2\otimes x^2 + x^2\otimes Q^2 + 2Qx \otimes Qx\\
  &= 2Qx \otimes Qx.
\end{align*}
Since $k$ is a field of characteristic not $2$, the only way for an element $a\in A$ to have the property that $2a\otimes a\in A\otimes A$ is zero is for $a$ itself to be zero. So we must have $Qx=0$.
\item If $Q\in G$ has an inverse $Q^{-1}$, then for each $x\in V_Q$, we have $x = Q^{-1}(Qx) = Q^{-1}\cdot 0 = 0$.
\end{enumerate}
\end{proof}

\begin{definition}\label{def of ZL}
Let $k$ be a field, let $G$ be a commutative monoid, and let $L$ be a $G$-suspensive Lie $k$-bialgebra. By the {\em universal left-sided $G$-rigid enveloping algebra of $L$,} written $ZL$, we mean the $k$-algebra which is the quotient of $WL$ by the two-sided ideal $I$ generated by:
\begin{itemize}
\item all products of the form $\ell_1\ell_2$, where $\ell_1\in GP_{Q_1}(WL)$ and $\ell_2\in GP_{Q_2}(WL)$ and $Q_1$ divides $Q_2$ in $G$,
\item and all products of the form $Q\ell$, where $Q\in G$ and $\ell\in GP_Q(WL)$.
\end{itemize}
\end{definition}
While the vanishing of the products $\ell_1\ell_2$ in Definition \ref{def of ZL} is clearly necessary in order for $ZL$ to be a left-sided bialgebra, the vanishing of the products $Q\ell$ is also necessary, due to Proposition \ref{left-sided products of grouplikes and gps}. (Specifically, if we fail to include the products $Q\ell$ in the ideal $I$, then the resulting ideal would fail to be a bi-ideal, so the resulting quotient $ZL/I$ would fail to be a bialgebra.)

\begin{prop}\label{ZL is a bialgebra}
Let $k$ be a field, let $G$ be a commutative monoid, and let $L$ be a $G$-suspensive Lie $k$-bialgebra. The two-sided ideal $I$ of $WL$, defined in Definition \ref{def of ZL}, is a bi-ideal in $WL$. Consequently $ZL$ is a $k$-bialgebra, and $WL\rightarrow ZL$ is a surjective $k$-bialgebra morphism.
\end{prop}
\begin{proof}
For $I$ to be a bi-ideal, it must satisfy
$\Delta(I) \subseteq I \otimes WL + WL \otimes I$ and $\epsilon(I)=0$.  
Suppose $\ell_1\in GP_{Q_1}(WL)$ and $\ell_2\in GP_{Q_2}(WL)$ and $Q_1$ divides $Q_2$ in $G$. Then there exists a $g\in G$ such that $Q_2=gQ_1$. Consequently we have
\begin{align*}\Delta(\ell_1 \ell_2) &= \Delta(\ell_1)\Delta(\ell_2) \\
%&=(\ell_1 \otimes Q_1 + Q_1\otimes \ell_1)(\ell_2 \otimes gQ_1 + gQ_1\otimes \ell_2)\\
&=\ell_1\ell_2\otimes Q_1gQ_1 + \ell_1 gQ_1 \otimes Q_1\ell_2 + Q_1\ell_2\otimes \ell_1gQ_1 + Q_1gQ_1 \otimes \ell_1\ell_2\\
&\subseteq I\otimes WL + WL\otimes I ,\end{align*}
since $\ell_1\ell_2$ and $Q_1\ell_1$ are each in $I$.
Similarly, if $\ell\in L_Q$, then:
\begin{align*}
 \Delta(Q\ell) 
%  &= (Q\otimes Q)(\ell\otimes Q + Q\otimes \ell) \\
  &= Q\ell\otimes Q^2 + Q^2\otimes Q\ell  \\
  &\subseteq I\otimes WL + WL\otimes I,
\end{align*}
since $Q\ell\in I$. So $\Delta(I)\subseteq WL\otimes I+I\otimes WL$.

As for the augmentation, since $\epsilon(\ell)=0$ for all $\ell \in L_Q$, and since $\epsilon$ is an algebra homomorphism, we have that $\epsilon$ vanishes on the generators of the ideal $I$, hence on all elements of $I$, as desired.
\end{proof}

Since $ZL$ is a quotient of $WL$ by elements contained in positive Lie filtration, the map $WL\rightarrow ZL$ is an isomorphism in Lie filtration $0$, i.e., an isomorphism on grouplikes. So, since $WL$ is $G$-rigid, the $k$-bialgebra $ZL$ is also $G$-rigid. In fact, under a certain reasonable hypothesis on the monoid $G$, the bialgebra $ZL$ enjoys much stronger properties, proven below in Proposition \ref{ZL is left-sided and GPG}.

The associated graded bialgebra $E_0R$ of the $K$-adic filtration on the Dyer-Lashof algebra $R$ is $\mathbb{N}$-rigid. The commutative monoid $\mathbb{N}$ has the property that its divisibility ordering is total. That is, given two elements $Q,Q^{\prime}$ of $\mathbb{N}$, either $Q$ divides $Q^{\prime}$ or $Q^{\prime}$ divides $Q$. (This statement looks strange because of course it is not true that, given two nonnegative integers, one is necessarily a divisor of the other. This strangeness is an artifact of mixing additive and multiplicative notation. For the sake of the terminology of Definition \ref{def of left-sided}, it is better to regard $\mathbb{N}$ as the free monoid on one generator, and to use multiplicative notation for it.) Hence $\mathbb{N}$ is {\em linear} in the following sense:
\begin{definition}\label{def of linear}
We say that a commutative monoid $G$ is {\em linear} if, for each pair $Q,Q^{\prime}\in G$, either $Q$ divides $Q^{\prime}$ or $Q^{\prime}$ divides $Q$.
\end{definition}

\begin{lemma}\label{wl to zl is surj on gp}
Let $k,G,L$ be as in Proposition \ref{ZL is a bialgebra}. Suppose that $G$ is linear. Then the map of Lie algebras $GP_*(WL)\rightarrow GP_*(ZL)$ is surjective.
\end{lemma}
\begin{proof}
By the linearity of $G$, given any pair of elements $\ell_1,\ell_2$ in Lie degree $1$ in $WL$, either the degree of $\ell_1$ divides the degree of $\ell_2$, or vice versa. Without loss of generality, assume that the degree of $\ell_1$ divides that of $\ell_2$. Then, in the quotient $ZL$ of $WL$, we have $\ell_1\ell_2 = 0$ and consequently $\ell_2\ell_1 = [\ell_2,\ell_1]$. Hence, in $ZL$, a product of elements in Lie filtration $1$ is also in Lie filtration $1$, i.e., the Lie filtration of $ZL$ collapses after the first stage. Hence every element of $ZL$ is a linear combination of grouplikes and grouplikes multiplied by elements of $L$. In particular, the generalized primitives of $ZL$---those elements of Lie filtration $1$ which do not lie in Lie filtration $0$---are in the image of the composite map $L \rightarrow GP_*(WL)\rightarrow GP_*(ZL)$, hence in the image of the map $GP_*(WL)\rightarrow GP_*(ZL)$.
\end{proof}

\begin{prop}\label{ZL is left-sided and GPG}
Let $k,G,L$ be as in Proposition \ref{ZL is a bialgebra}. Suppose that $G$ is linear. Then the $G$-rigid $k$-bialgebra $ZL$ is left-sided and generalized-primitively-generated.
\end{prop}
\begin{proof}
Let $Q,Q^{\prime}\in G$, suppose that $Q$ divides $Q^{\prime}$, and suppose that $x\in GP_Q(ZL)$ and $y\in GP_{Q^{\prime}}(ZL)$. Use Lemma \ref{wl to zl is surj on gp} to lift $x$ to an element $\tilde{x}\in GP_Q(WL)$, and lift $y$ to an element $\tilde{y}\in GP_{Q^{\prime}}(WL)$. Then $\tilde{x}\tilde{y}$ is in the kernel of $WL\rightarrow ZL$. So $xy=0$, i.e., $ZL$ is left-sided.

As for the claim of generalized-primitive generation: since $WL$ is generalized-primitively-generated, every quotient bialgebra of $WL$ is generalized-primitively-generated. In particular, $ZL$ is generalized-primitively-generated.
\end{proof}

\begin{comment} (UNFINISHED, NOT CURRENTLY USED)
\begin{lemma}
Let $G$ be a commutative monoid, let $k$ be a field, and let $L$ be a torsion $G$-suspensive Lie $k$-algebra. Let $I$ be the two-sided ideal in $WL$ defined in Definition \ref{def of ZL}. Then the intersection of $I$ with $L\subseteq WL$ is $\{ 0\}\subseteq WL$.
\end{lemma}
\begin{proof}
Suppose that $Q_1,Q_2\in G$, suppose that $\ell_1\in L_{Q_1}$ and $\ell_2\in L_{Q_2}$, and suppose that $\ell_1\ell_2$ is in the image of the map $L\rightarrow WL$. Then $\ell_1\ell_2$ is a generalized primitive in $WL$, so we must have
\begin{align}
\label{eq j0438} Q_1Q_2\otimes \ell_1\ell_2 + \ell_1\ell_2\otimes Q_1Q_2 
  &= \Delta(\ell_1\ell_2) \\
\nonumber  &= (Q_1\otimes \ell_1 + \ell_1\otimes Q_1)(Q_2\otimes \ell_2 + \ell_2\otimes Q_2) \\
\label{eq j0439}  &= Q_1Q_2\otimes \ell_1\ell_2 + \ell_1\ell_2\otimes Q_1Q_2  + Q_1\ell_2\otimes Q_2\ell_1 + Q_2\ell_1\otimes Q_1\ell_2 .
\end{align}
Applying the product in $WL$ to the left-hand side of \eqref{eq j0438} and to \eqref{eq j0439}, we have
\begin{align}
\nonumber 2Q_1Q_2\ell_1\ell_2
  &= 3 Q_1Q_2 \ell_1\ell_2 + Q_1Q_2\ell_2\ell_1 ,\mbox{\ \ i.e.,}\\
\label{eq j0440} 0 &= Q_1Q_2[\ell_1,\ell_2].
\end{align}
Since $L$ is torsion, \eqref{eq j0440} implies that $[\ell_1,\ell_2]=0$ in $L$.  NOPE! ...
\end{proof}
\end{comment}

\begin{theorem}\label{mm thm for left-sided bialgs}
Let $G$ be a {\em linear} commutative monoid, and let $k$ be a field of characteristic zero. Write $\TSusp_G\Lie(k)$ for the category of torsion $G$-suspensive Lie $k$-algebras, and write $\LeftSided\GPGen\Rig_G\Bialg(k)$ for the category of left-sided generalized-primitively-generated $G$-rigid $k$-bialgebras. 
Then the functors 
\begin{align*} GP_*: \LeftSided\GPGen\Rig_G\Bialg(k) &\rightarrow \TSusp_G\Lie(k)\mbox{\ \ \ \ and} \\
 Z: \TSusp_G\Lie(k) &\rightarrow \LeftSided\GPGen\Rig_G\Bialg(k)\end{align*}
are mutually inverse. Consequently the category of torsion $G$-suspensive Lie $k$-algebras is equivalent to the category of left-sided generalized-primitively-generated $G$-rigid $k$-bialgebras.
\end{theorem}
\begin{proof}
Given a $G$-suspensive Lie $k$-algebra $L$, we have the composite map of $k$-bialgebras $L \hookrightarrow GP_*(WL) \rightarrow GP_*(ZL)$. 
Write $f$ for this composite map $L\rightarrow GP_*(ZL)$. If $L$ is torsion, then we claim that $f$ is an isomorphism. The proof is as follows:
\begin{description}
\item[Injectivity of $f$] 
\begin{comment} (BETTER ARGUMENT BELOW)
The algebra $ZL$ is obtained by imposing upon $WL$ the relations:
\begin{itemize}
\item $\ell_1\ell_2 =0$ whenever $\ell_1\in GP_{Q_1}(WL)$ and $\ell_2\in GP_{Q_2}(WL)$ and $Q_1$ divides $Q_2$ in $G$, 
\item and $Q\ell =0$, where $Q\in G$ and $\ell\in GP_Q(WL)$.
\end{itemize}
The second family of relations already hold in $WL$, since $L$ is assumed to be a torsion Lie algebra\footnote{This is the only place in this proof in which the assumption that $L$ is torsion plays a role.}. Imposing the first family of relations simply identifies the bracket $[\ell_1,\ell_2] \in L\subseteq WL$ with the product $\ell_1\ell_2$ if the degree of $\ell_2$ divides the degree of $\ell_1$, and instead\footnote{One of these two possibilities indeed holds, since $G$ is linear. This is the only place in this proof in which the assumption that $G$ is linear plays a role.} with the product $-\ell_2\ell_1$ if the degree of $\ell_1$ divides the degree of $\ell_1$. Consequently $ZL$ in positive Lie filtrations agrees with $L$ itself, i.e., $f$ is injective. 
\end{comment}
We have the commutative diagram of $k$-vector spaces
\[\xymatrix{
 L \ar[r] 
  & \LieFilt UL \ar[r]\ar[d]^{\cong} & \SLieFilt WL \ar[d]^{\cong} \ar[r]& \SLieFilt ZL \ar[d] \\
  & \LieFilt U(L^{\#}) \ar[r]\ar[d]^{\cong} & \SLieFilt W(L^{\tilde{\#}}) \ar[d]^{\cong} \ar[r] & \SLieFilt Z(L^{\tilde{\#}}) \ar[d]^{\cong} \\
  & S_k(L^{\#}) \ar[r] & S_{kG}(L^{\tilde{\#}}) \ar[r] & kG\oplus L^{\tilde{\#}}.
}\]
The composite $L \rightarrow kG\oplus L^{\tilde{\#}}$ is an isomorphism of $k$-vector spaces onto Lie degree $1$ in $kG\oplus L^{\tilde{\#}}$. So $L\rightarrow \SLieFilt ZL$ composed with a map $\SLieFilt ZL \rightarrow kG\oplus L^{\tilde{\#}}$ is injective, so  $L\rightarrow \SLieFilt ZL$ is injective, and consequently $L\rightarrow ZL$ is injective.
\item[Surjectivity of $f$]  %By Proposition \ref{GP and surjectivity}, every generalized primitive in $ZL$ is in the image of the map $WL\rightarrow ZL$. 
This is a consequence of the argument given in the proof of Lemma \ref{wl to zl is surj on gp}.
\end{description}
So $Z: \Susp_G\Lie(k)\rightarrow \LeftSided\GPGen\Rig_G\Bialg(k)$ is right inverse to $GP_*$.

Now suppose that $A$ is a left-sided $G$-rigid $k$-bialgebra. Then $Z(GP_*A)$ is the free left-sided $G$-rigid $k$-bialgebra on the $G$-suspensive Lie $k$-algebra $GP_*A$, so we have a canonical map $g: Z(GP_*A)\rightarrow A$. If $A$ is furthermore assumed to be generalized-primitively-generated, then $Z(GP_*A)$ is the free left-sided rigid bialgebra on a set of generators for $A$, so $g: Z(GP_*A)\rightarrow A$ is surjective. Since $GP_*(g): GP_*(Z(GP_*A))\rightarrow GP_*(A)$ is an isomorphism by the previous part of this theorem, it is in particular injective, so by Proposition \ref{GP and injectivity}, $g$ itself is injective. So $g$ is an isomorphism. So $Z: \Susp_G\Lie(k)\rightarrow \LeftSided\GPGen\Rig_G\Bialg(k)$ is also left inverse to $GP_*$.
\end{proof}

\appendix
\section{Review of the Dyer-Lashof algebra.}
\label{Review of the Dyer-Lashof algebra.}
% Added \sloppy to take care of some of the formatting issues. YDP
\sloppy

The material in this appendix is classical; see Theorem 1.1 of Steinberger's chapter ``Homology operations for $H_{\infty}$ and $H_n$ ring spectra'' in \cite{MR836132}, or the first chapter of \cite{MR0436146}, for example.

We recall a presentation for the Dyer-Lashof algebra. If $p = 2$, let $R(-\infty)$ denote the free associative graded $\mathbb{F}_2$-algebra on generators $Q^0, Q^1, Q^2, \dots$, with $Q^i$ in grading degree $i$, modulo the Adem relation
\begin{align*}
    Q^r Q^s &= \sum_i \binom{i-s-1}{2i-r} Q^{r+s-i} Q^i
\end{align*}
for all $r > 2s$.

For an odd prime $p$, let $R(-\infty)$ denote instead the free associative graded $\mathbb{F}_p$ algebra on generators $Q^0, Q^1, Q^2, \ldots$ and $\beta Q^0, \beta Q^1, \beta Q^2, \ldots$ with $Q^i$ in grading degree $2i(p-1)$ and $\beta Q^i$ in grading degree $2i(p-1)-1$, modulo the Adem relations
\begin{align*}
  Q^r Q^s &= \sum_i (-1)^{r+i} \binom{pi-(p-1)s-i-1}{pi-r} Q^{r+s-i} Q^i, \\
  Q^r \beta Q^s &= \sum_i (-1)^{r+i} \binom{pi-(p-1)s-i}{pi-r} \beta Q^{r+s-i} Q^i \\
  &-\sum_i (-1)^{r+i} \binom{pi-(p-1)s-i-1}{pi-r-1} Q^{r+s-i} \beta Q^i, \\
  \beta Q^r\beta Q^s &= -\sum_i (-1)^{r+i} \binom{pi-(p-1)s-i-1}{pi-r-1} \beta Q^{r+s-i} \beta Q^i.
\end{align*}
for all $r > ps$. 

When $p=2$, a useful notational convention which is sometimes used (e.g. in \cite{MR711048}) is to write
\begin{itemize}
\item $Q^{r}$ rather than $Q^{2r}$ for the generator for $R(-\infty)$ in degree $2r$, and
\item $\beta Q^{r}$ rather than $Q^{2r-1}$ for the generator for $R(-\infty)$ in degree $2r-1$. 
\end{itemize}
With these conventions, the Adem relations and degrees for the generators of $R(-\infty)$ at the prime $2$ are the same as the Adem relations and degrees for the generators of $R(-\infty)$ at odd primes.

Let $\beta^{\epsilon_1} Q^{i_1} \beta^{\epsilon_2} Q^{i_2} \cdots \beta^{\epsilon_d} Q^{i_d}$ be a monomial in this free associative graded algebra. %This monomial is said to be {\em admissible} if $2i_{j+1} > i_j$ for all $j = 1, \dots, d-1$ when $p = 2$, or if $pi_{j+1} - \epsilon_{j+1} \geq i_j$ for all $j = 1, \dots, d-1$ when $p$ is odd. 
The {\em excess} of this monomial is defined to be $i_1 - \sum_{j=2}^d i_j$ when $p = 2$ and
\[ 2i_1 - \epsilon_1 - \sum_{j=2}^d \left(2 i_j (p-1) - \epsilon_j \right) \]
when $p$ is odd. For an integer $e$, let $J_e$ be the two-sided ideal of $R(-\infty)$ generated by all monomials of excess $< e$. We let $R(e) = R(-\infty)/J_e$. The special case $e=0$ is called the \emph{Dyer-Lashof algebra}. We often write $R$ for $R(0)$.

The coproduct and augmentation on $R$ are given by
\begin{align*}
 \Delta(Q^n) &= \sum_{j=0}^n Q^j\otimes Q^{n-j},\\
 \Delta(\beta Q^{n+1}) &= \sum_{j=0}^n \left( \beta Q^{j+1}\otimes Q^{n-j} + \beta Q^{j}\otimes Q^{n+1-j}\right),\\
 \epsilon(Q^0) &= 1,\\
 \epsilon(Q^n) &= 0\mbox{\ \ if\ } n>0.
\end{align*}
A nice reference for the coproduct is Theorem 2.3 in the first chapter of \cite{MR0436146}. Since $Q^0$ is a grouplike but not invertible, one sees immediately that $R$ cannot be a Hopf algebra.

\bibliography{salch}{}
\bibliographystyle{plain}
\end{document}